\documentclass[final,3p]{elsarticle}
\usepackage[latin1]{inputenc}
 \usepackage{graphics}
 \usepackage{graphicx}
 \usepackage{epsfig}
\usepackage{amssymb}
 \usepackage{amsthm}
 \usepackage{lineno}
 \usepackage{amsmath}
\usepackage{color}
   \numberwithin{equation}{section}
\usepackage{mathrsfs}

\journal{ } 

\newtheorem{thm}{Theorem}[section]
\newtheorem{cor}[thm]{Corollary}
\newtheorem{lem}[thm]{Lemma}

\biboptions{sort&compress,square}
\allowdisplaybreaks
\begin{document}
\begin{frontmatter}
\author{Tong Wu$^{a,b}$}
\ead{wut977@nenu.edu.cn}
\author{Yong Wang$^{c,*}$}
\ead{wangy581@nenu.edu.cn}
\cortext[cor]{Corresponding author.}
\address{$^a$Department of Mathematics, Northeastern University, Shenyang, 110819, China\\
$^b$Key Laboratory of Data Analytics and Optimization for Smart Industry, Ministry of Education,\\ Northeastern University, Shenyang, Liaoning 110819, China}
\address{$^c$School of Mathematics and Statistics, Northeast Normal University,
Changchun, 130024, China}

\title{The Connes-Chamseddine cycle and the noncommutative integral}
\begin{abstract}
In \cite{Co3}, Connes and Chamseddine defined a cycle in the general framework of noncommutative geometry. They computed this cycle for the Dirac operator on 4-dimensioanl manifolds. We propose a way to study the Connes-Chamseddine cycle from the viewpoint of the noncommutative integral on 6-dimensional manifolds in this paper. Furthermore, we compute several interesting noncommutative integral defined in \cite{Fi} by the normal coodinated way on n-dimensional manifolds. As a corollary, the Connes-Chamseddine cycle on 6-dimensional manifolds is obtained.
\end{abstract}
\end{frontmatter}
\section{Introduction}
 Until now, many geometers have studied noncommutative residues. In \cite{Gu,Wo}, authors found noncommutative residues are of great importance to the study of noncommutative geometry. Connes showed us that the noncommutative residue on a compact manifold $M$ coincided with the Dixmier's trace on pseudodifferential operators of order $-{\rm {dim}}M$ in \cite{Co2}. Therefore, the non-commutative residue can be used as integral of noncommutative geometry and become an important tool of noncommutative geometry. In \cite{Co1}, Connes used the noncommutative residue to derive a conformal 4-dimensional Polyakov action analogy.
Several years ago, Connes made a challenging observation that the noncommutative residue of the square of the inverse of the Dirac operator was proportional to the Einstein-Hilbert action, which we call the Kastler-Kalau-Walze theorem. In \cite{Ka}, Kastler gave a bruteforce proof of this theorem. In \cite{KW}, Kalau and Walze proved this theorem in the normal coordinates system simultaneously. Ackermann proved that the Wodzicki residue  of the square of the inverse of the Dirac operator ${\rm Wres}(D^{-2})$ in turn is essentially the second coefficient of the heat kernel expansion of $D^{2}$ in \cite{Ac}.

The many noncommutative residues of Dirac operators have been studied \cite{W1,W3,W4,D1}. Wang proved a Kastler-Kalau-Walze type theorem for perturbations of Dirac operators on compact manifolds with or without boundary and gave two kinds of operator-theoretic explanations of the gravitational
action on boundary in \cite{W1}. In \cite{W4}, Wang, Wang and Yang gave two kinds of operator-theoretic
explanations of the gravitational action about Dirac operators with torsion in the case of 4-dimensional compact manifolds with
flat boundary. In \cite{W3}, Wang, Wang and Wu gave some new
spectral functionals which is the extension of spectral functionals to the noncommutative realm with torsion, and related them to the
noncommutative residue for manifolds with boundary about Dirac operators with torsion. In \cite{D1},  Dabrowski, Sitarz and Zalecki recovered the Einstein functional by Dirac operators and the noncommutative residue. In \cite{Co3}, Connes and Chamseddine defined a cycle which contains the Dirac operator in the general framework of noncommutative geometry. They proved that the inner fluctuations of the spectral action can be computed as residues and give exactly the counterterms for the Feynman graphs with fermionic internal lines. And they computed this cycle for the Dirac operator on 4-dimensioanl manifolds. In \cite{Fi}, Figueroa H, Gracia-Bondia J M, Lizzi F, et al. introduced a noncommutative integral based on the noncommutative residue \cite{Wo1}. $\mathbf{The~motivation}$ of this paper is to compute several interesting noncommutative integral about some operators from \cite{Co3} by the normal coodinated way on n-dimensional manifolds.

 The organization of this paper is as follows. In Section \ref{section:2}, we recall some basic facts and formulas about Dirac operator and we introduce the Connes-Chamseddine cycle. In Section \ref{section:3}, using the relationship between the noncommutative integral and the noncommutative residue, that is the formula $\int\hspace{-1.05em}- Pds^n:={\rm Wres} (PD^{-n}):=\int_{S^*M}{\rm tr}[\sigma_{-n}(PD^{-n})](x,\xi)$, we give the noncommutative integral of the Connes-Chamseddine cycle on n-dimensional manifolds without boundary. To do it, we compute the corresponding symbols for the pseudo-differential operators, and then get trace of the results in the normal coordinate system. And finally, we obtain the Connes-Chamseddine cycle defined by in \cite{Co3} on 6-dimensional manifolds with boundary.
\section{The Dirac operator and the Connes-Chamseddine cycle}
\label{section:2}
Let $M$ be an n-dimensional compact oriented spin manifold with Riemannian metric $g$, and let $\nabla^L$ be the Levi-Civita connection about $g$.
We recall that the Dirac operator $D$ is locally given as follows in terms of an orthonormal section $e_i$ (with dual section $\theta^k$) of the frame bundle of M \cite{Ka}:
\begin{align}
&D=i\gamma^i\widetilde{\nabla}_i=i\gamma^i(e_i+\sigma_i);\nonumber\\
&\sigma_i(x)=\frac{1}{4}\gamma_{ij,k}(x)\gamma^i\gamma^k=\frac{1}{8}\gamma_{ij,k}(x)[\gamma^j\gamma^k-\gamma^k\gamma^j],
\end{align}
where $\gamma_{ij,k}$ represents the Levi-Civita connection $\nabla$ with spin connection $\widetilde{\nabla}$, specifically:
\begin{align}
&\gamma_{ij,k}=-\gamma_{ik,j}=\frac{1}{2}[c_{ij,k}+c_{ki,j}+c_{kj,i}],~~~i,j,k=1,\cdot\cdot\cdot,4;\nonumber\\
&c_{ij}^k=\theta^k([e_i.e_j]).
\end{align}
Here the $\gamma^i$ are constant self-adjoint Dirac matrices s.t. $\gamma^i\gamma^j+\gamma^j\gamma^i=-2\delta^{ij}.$ In terms
of local coordinates $x^\mu$ inducing the alternative vierbein $\partial_\mu=S_\mu^i(x)e_i$ (with dual
vierbein $dx^\mu$) we have $\gamma^ie_i=\gamma^\mu \partial_\mu$, the $\gamma^\mu$  being now $x$-dependent Dirac matrices s.t. $\gamma^\mu\gamma^\nu+\gamma^\nu\gamma^\mu=-2g^{\mu\nu}$ (we use latin sub-(super-) scripts for the basic $e_i$ and greek
sub-(super-) scripts for the basis $\partial_\mu$, the type of sub-(super-) scripts specifying the type
of Dirac matrices). The specification of the Dirac operator in the greek basis is as
follows: one has
\begin{align*}
&D=i\gamma^\mu\widetilde{\nabla}_\mu=i\gamma^\mu(e_\mu+\sigma_\mu);\nonumber\\
&\sigma_\mu(x)=S_\mu^i(x)\sigma_i.
\end{align*}
Then, using the shorthand: $\Gamma^{k}=g^{ij}\Gamma_{ij}^{k};~~\sigma^{j}=g^{ij}\sigma_{i}$, we have respective symbols of $D^{2}$:
\begin{lem}\cite{Ka}\label{lemma1}
\begin{align}
\sigma_2(D^2)&=|\xi|^2;\nonumber\\
\sigma_1(D^2)&=i(\Gamma^\mu-2\sigma^\mu)(x)\xi_\mu;\nonumber\\
\sigma_0(D^2)&=-g^{\mu\nu}(\partial^{x}_\mu\sigma_\nu+\sigma^\mu\sigma_\nu-\Gamma^\alpha_{\mu\nu}\sigma_\alpha)(x)+\frac{1}{4}s(x).
\end{align}
\end{lem}
We now introduce the Connes-Chamseddine cycle. Firstly, the spectral action is defined as a functional on noncommutative geometries. Such a geometry is specified by a
fairly simple data of operator theoretic nature, namely a spectral triple
$$(\mathcal{A},\mathcal{H},D)$$
where $\mathcal{A}$ is a noncommutative algebra with involution, acting in the Hilbert space $\mathcal{H}$ while $D$ is a Dirac operator, which is self-adjoint
operator with compact resolvent and such that
$[D,a]$ is bounded $\forall a \in \mathcal{A}$. Then $\rho$ is a cycle which is given by the following formula defined by Connes and Chamseddine in \cite{Co3},
\begin{align}\label{a2}
\rho(a^0,a^1,a^2,a^3,a^4)&=\int\hspace{-1.05em}-a^0[D^2,a^1][D^2,a^2][D,a^3][D,a^4]D^{-6}-\int\hspace{-1.05em}-a^0[D,a^1][D^2,a^2][D^2,a^3][D,a^4]D^{-6}\nonumber\\
&+\int\hspace{-1.05em}-a^0[D,a^1][D,a^2][D^2,a^3][D^2,a^4]D^{-6}-\int\hspace{-1.05em}-a^0[D^2,a^1][D,a^2][D,a^3][D^2,a^4]D^{-6},
\end{align}
where $\int\hspace{-1.05em}-$ denotes the noncommutative residue.
\section{The noncommutative integral}
\label{section:3}
For a pseudo-differential operator $P$, acting on sections of a spinor bundle over an even n-dimensional compact
Riemannian spin manifold $M$, the analogue of the volume element in noncommutative geometry is theoperator $D^{-n}=:ds^n$. And pertinent operators are realized as pseudodifferential operators on the spaces of sections. Extending previous definitions by Connes \cite{Co4}, a
noncommutative integral was introduced in \cite{Fi} based on the noncommutative residue \cite{Wo1}, combine (1.4) in \cite{Co3} and \cite{Ka}, using the definition of the residue:
\begin{align}\label{666}
\int\hspace{-1.05em}- Pds^n:={\rm Wres}PD^{-n}:=\int_{S^*M}{\rm tr}[\sigma_{-n}({PD^{-n}})](x,\xi),
\end{align}
where $\sigma_{-n}(PD^{-n})$ denotes the ($-n$)th order piece of the complete symbols of $PD^{-n}$, {\rm tr} as shorthand of trace.

By (\ref{666}), we find the Connes-Chamseddine cycle consists of four noncommutative integrals mentioned above. Nextly, in order to get the result of (\ref{a2}), we need to divide the Connes-Chamseddine cycle into four parts and study them separately.

$\mathbf{Part~~i)}$ $P=[D^2,a^1][D^2,a^2][D,a^3][D,a^4],$ and $[D^2,a^1]:=A,~[D^2,a^2][D,a^3][D,a^4]:=B$.\\
Let $n=2m$, by (\ref{666}), we need to compute $\int_{S^*M}{\rm tr}[\sigma_{-2m}(ABD^{-2m})](x,\xi).$ Based on the algorithm yielding the principal
symbol of a product of pseudo-differential operators in terms of the principal symbols of the factors, we have
\begin{align}
\sigma_{-2m}(ABD^{-2m})&=\left\{\sum_{|\alpha|=0}^\infty\frac{(-i)^{|\alpha|}}{\alpha!}\partial^\alpha_\xi[\sigma(AB)]\partial^\alpha_x[\sigma(D^{-2m})]\right\}_{-2m}\nonumber\\
&=\sigma_0(AB)\sigma_{-2m}(D^{-2m})+\sigma_1(AB)\sigma_{-2m-1}(D^{-2m})+\sigma_2(AB)\sigma_{-2m-2}(D^{-2m})\nonumber\\
&+(-i)\sum_{j=1}^n\partial_{\xi_j}[\sigma_2(AB)]\partial_{x_j}[\sigma_{-2m-1}(D^{-2m})]+(-i)\Sigma_{j=1}^n\partial_{\xi_j}[\sigma_1(AB)]\partial_{x_j}[\sigma_{-2m}(D^{-2m})]\nonumber\\
&-\frac{1}{2}\sum_{jl}\partial_{\xi_j}\partial_{\xi_l}[\sigma_2(AB)]\partial_{x_j}\partial_{x_l}[\sigma_{-2m}(D^{-2m})].
\end{align}
\begin{lem}\cite{UW}\label{lema2}
Let $S$ be pseudo-differential operator of order $k$ and $f$ is a smooth function, $[S,f]$ is a pseudo-differential operator of order $k-1$ with total symbol
$\sigma[S,f]\sim\sum_{j\geq 1}\sigma_{k-j}[S,f]$, where
\begin{align}
\sigma_{k-j}[S,f]=\sum_{|\beta|=1}^j\frac{D_x^\beta(f)}{\beta!}\partial_\xi^\beta(\sigma^S_{k-(j-|\beta|)}).
\end{align}
\end{lem}
By Lemma \ref{lemma1} and Lemma \ref{lema2}, we get the following lemma.
\begin{lem}The symbols of $A$ and $B$ are given
\begin{align}
&\sigma_0(A)=\sum_{j=1}^n\partial_{x_j}(a^1)(\Gamma^j-2\sigma^j)(x)-\sum_{jl=1}^n\partial_{x_j}\partial_{x_l}(a^1)g^{jl}\nonumber;\\
&\sigma_1(A)=-2i\sum_{jl=1}^n\partial_{x_j}(a^1)g^{jl}\xi_l\nonumber;\\
&\sigma_0(B)=\bigg(\sum_{j=1}^n\partial_{x_j}(a^2)(\Gamma^j-2\sigma^j)(x)-\sum_{jl=1}^n\partial_{x_j}\partial_{x_l}(a^2)g^{jl}\bigg)c(da^3)c(da^4)-2\sum_{jl=1}^n\partial_{x_l}(a^2)g^{jl}\partial_{x_j}[c(da^3)c(da^4)]\nonumber;\\
&\sigma_1(B)=-2i\sum_{jl=1}^n\partial_{x_j}(a^2)g^{jl}\xi_lc(da^3)c(da^4).
\end{align}
\end{lem}
Further, by the composition formula of pseudodifferential operators, we get the following lemma.
\begin{lem}The symbols of $AB$ are given
\begin{align}
\sigma_0(AB)&=\sigma_0(A)\sigma_0(B)+(-i)\partial_{\xi_j}[\sigma_1(A)]\partial_{x_j}[\sigma_0(B)]+(-i)\partial_{\xi_j}[\sigma_0(A)]\partial_{x_j}[\sigma_1(B)]\nonumber\\
&=\bigg(\sum_{k=1}^n\partial_{x_k}(a^1)(\Gamma^k-2\sigma^k)(x)-\sum_{kl=1}^n\partial_{x_k}\partial_{x_l}(a^1)g^{jl}\bigg)\times\bigg[\bigg(\sum_{j=1}^n\partial_{x_j}(a^2)(\Gamma^j-2\sigma^j)(x)\nonumber\\
&-\sum_{jl=1}^n\partial_{x_j}\partial_{x_l}(a^2)g^{jl}\bigg)c(da^3)c(da^4)-2\sum_{jl=1}^n\partial_{x_l}(a^2)g^{jl}\partial_{x_j}[c(da^3)c(da^4)]\bigg]-2\sum_{jl=1}^n\bigg\{\partial_{x_j}(a^1)g^{jl}\nonumber\\
&\bigg[\bigg(\sum_{r=1}^n\partial_{x_j}\partial_{x_r}(a^2)(\Gamma^r-2\sigma^r)(x)+\sum_{r=1}^n\partial_{x_r}(a^2)\partial_{x_j}(\Gamma^r-2\sigma^r)(x)-\sum_{rs=1}^n\partial_{x_j}\partial_{x_r}\partial_{x_s}(a^2)g^{rs}\nonumber\\
&-\sum_{rs=1}^n\partial_{x_r}\partial_{x_s}(a^2)\partial_{x_j}g^{rs}\bigg)c(da^3)c(da^4)+\bigg(\sum_{j=1}^n\partial_{x_j}(a^2)(\Gamma^j-2\sigma^j)(x)-\sum_{jl=1}^n\partial_{x_j}\partial_{x_l}(a^2)g^{jl}\bigg)\nonumber\\
&\partial_{x_j}[c(da^3)c(da^4)]-2\sum_{pl=1}^n\partial_{x_j}\partial_{x_l}(a^2)g^{pl}\partial_{x_p}[c(da^3)c(da^4)]-2\sum_{pl=1}^n\partial_{x_l}(a^2)g^{pl}\partial_{x_j}\partial_{x_p}[c(da^3)c(da^4)]\bigg]\nonumber\\
&-2\sum_{jlp}\partial_{x_l}(a^2)\partial_{x_j}(g^{pl})\partial_{x_p}[c(da^3)c(da^4)]\bigg\}\nonumber;\\
\sigma_1(AB)&=\sigma_1(A)\sigma_0(B)+\sigma_0(A)\sigma_1(B)+(-i)\partial_{\xi_j}[\sigma_1(A)]\partial_{x_j}[\sigma_1(B)]\nonumber\\
&=-2i\sum_{pk=1}^n\partial_{x_p}(a^1)g^{pk}\xi_p\bigg(\sum_{r=1}^n\partial_{x_r}(a^2)(\Gamma^r-2\sigma^r)(x)-\sum_{rs=1}^n\partial_{x_r}\partial_{x_s}(a^2)g^{rs}\bigg)c(da^3)c(da^4)\nonumber\\
&+4i\sum_{pk=1}^n\partial_{x_p}(a^1)g^{pk}\xi_p\bigg(\sum_{ql=1}^n\partial_{x_l}(a^2)g^{ql}\partial_{x_q}[c(da^3)c(da^4)]\bigg)\nonumber\\
&-2i\bigg(\sum_{k=1}^n\partial_{x_k}(a^1)(\Gamma^k-2\sigma^k)(x)-\sum_{kl=1}^n\partial_{x_k}\partial_{x_l}(a^1)g^{kl}\bigg)\sum_{ts=1}^n\partial_{x_t}(a^2)g^{ts}\xi_tc(da^3)c(da^4)\nonumber\\
&+4i\sum_{j\alpha=1}^n\partial_{x_\alpha}(a^1)g^{\alpha j}\sum_{ts=1}^n\partial_{x_j}\partial_{x_t}(a^2)g^{ts}\xi_tc(da^3)c(da^4)\nonumber\\
&+4i\sum_{j\alpha=1}^n\partial_{x_\alpha}(a^1)g^{\alpha j}\sum_{ts=1}^n\partial_{x_t}(a^2)g^{ts}\xi_t\partial_{x_j}[c(da^3)c(da^4)]\nonumber\\
&+4i\sum_{j\alpha=1}^n\partial_{x_\alpha}(a^1)g^{\alpha j}\sum_{ts=1}^n\partial_{x_t}(a^2)\partial_{x_j}g^{ts}\xi_tc(da^3)c(da^4);\nonumber\\
\sigma_2(AB)&=\sigma_1(A)\sigma_1(B)\nonumber\\
&=-4\sum_{jlks=1}^n\partial_{x_j}(a^1)\partial_{x_l}(a^2)\xi_j\xi_lg^{jk}g^{ls}c(da^3)c(da^4).
\end{align}
\end{lem}
By Lemma \ref{lemma1}, (4.24) in \cite{KW} and (3.4) in \cite{Wa7}, we get
\begin{lem}\label{lemkkk}General dimensional symbols about Dirac operator $D$ are given,
\begin{align}
\sigma_{-2m}(D^{-2m})&=|\xi|^{-2m};\nonumber\\
\sigma_{-2m-1}(D^{-2m})&=m|\xi|^{-(2m-2)}\bigg(-i|\xi|^{-4}\xi_k(\Gamma^k-2\sigma^k)-2i|\xi|^{-6}\xi^j\xi_\alpha\xi_\beta\partial_{x_j}g^{\alpha\beta}\bigg)\nonumber\\
&+2i\sum_{k=0}^{m-2}\sum_{\mu=1}^{2m}(k+1-m)|\xi|^{-2m-4}\xi^\mu\xi_\alpha\xi_\beta\partial^x_\mu g^{\alpha\beta};\nonumber\\
\sigma_{-2m-2}(D^{-2m})(x_0)&=-\frac{m}{4}|\xi|^{-2m-2}s+\frac{m(m+1)}{3}|\xi|^{-2m-4}\xi_\mu\xi_\nu R_{\mu\alpha\nu\alpha}(x_0),
\end{align}
where $s$ is the scalar curvature.
\end{lem}
\begin{proof}
See the Appendix.
\end{proof}
Next, we review here technical tool of the computation, which are the integrals of polynomial functions over the unit spheres. By (32) in \cite{B1}, we define
\begin{align}
I_{S_n}^{\gamma_1\cdot\cdot\cdot\gamma_{2\bar{n}+2}}=\int_{|x|=1}d^nxx^{\gamma_1}\cdot\cdot\cdot x^{\gamma_{2\bar{n}+2}},
\end{align}
i.e. the monomial integrals over a unit sphere.
Then by Proposition A.2. in \cite{B1},  polynomial integrals over higher spheres in the $n$-dimesional case are given
\begin{align}
I_{S_n}^{\gamma_1\cdot\cdot\cdot\gamma_{2\bar{n}+2}}=\frac{1}{2\bar{n}+n}[\delta^{\gamma_1\gamma_2}I_{S_n}^{\gamma_3\cdot\cdot\cdot\gamma_{2\bar{n}+2}}+\cdot\cdot\cdot+\delta^{\gamma_1\gamma_{2\bar{n}+1}}I_{S_n}^{\gamma_2\cdot\cdot\cdot\gamma_{2\bar{n}+1}}],
\end{align}
where $S_n\equiv S^{n-1}$ in $\mathbb{R}^n$.

For $\bar{n}=0$, we have $I^0$=area$(S_n)$=$\frac{2\pi^{\frac{n}{2}}}{\Gamma(\frac{n}{2})}$, we immediately get
\begin{align}
I_{S_n}^{\gamma_1\gamma_2}&=\frac{1}{n}area(S_n)\delta^{\gamma_1\gamma_2};\nonumber\\
I_{S_n}^{\gamma_1\gamma_2\gamma_3\gamma_4}&=\frac{1}{n(n+2)}area(S_n)[\delta^{\gamma_1\gamma_2}\delta^{\gamma_3\gamma_4}+\delta^{\gamma_1\gamma_3}\delta^{\gamma_2\gamma_4}+\delta^{\gamma_1\gamma_4}\delta^{\gamma_2\gamma_3}];\nonumber\\
I_{S_n}^{\gamma_1\gamma_2\gamma_3\gamma_4\gamma_5\gamma_6}&=\frac{1}{n(n+2)(n+4)}area(S_n)[\delta^{\gamma_1\gamma_2}(\delta^{\gamma_3\gamma_4}\delta^{\gamma_5\gamma_6}+\delta^{\gamma_3\gamma_5}\delta^{\gamma_4\gamma_6}+\delta^{\gamma_3\gamma_6}\delta^{\gamma_4\gamma_5})\nonumber\\
&+\delta^{\gamma_1\gamma_3}(\delta^{\gamma_2\gamma_4}\delta^{\gamma_5\gamma_6}+\delta^{\gamma_2\gamma_5}\delta^{\gamma_4\gamma_6}+\delta^{\gamma_2\gamma_6}\delta^{\gamma_4\gamma_5})+\delta^{\gamma_1\gamma_4}(\delta^{\gamma_2\gamma_3}\delta^{\gamma_5\gamma_6}+\delta^{\gamma_2\gamma_5}\delta^{\gamma_3\gamma_6}+\delta^{\gamma_2\gamma_6}\delta^{\gamma_3\gamma_5})\nonumber\\
&+\delta^{\gamma_1\gamma_5}(\delta^{\gamma_2\gamma_3}\delta^{\gamma_4\gamma_6}+\delta^{\gamma_2\gamma_4}\delta^{\gamma_3\gamma_6}+\delta^{\gamma_2\gamma_6}\delta^{\gamma_3\gamma_4})+\delta^{\gamma_1\gamma_6}(\delta^{\gamma_2\gamma_3}\delta^{\gamma_4\gamma_5}+\delta^{\gamma_2\gamma_4}\delta^{\gamma_3\gamma_5}+\delta^{\gamma_2\gamma_5}\delta^{\gamma_3\gamma_4})].
\end{align}
Next, we calculate each term of  $\int_{S^*M}{\rm tr}[\sigma_{-2m}(ABD^{-2m})](x,\xi)$ separately.\\
$\mathbf{(i-1)}$ For $\sigma_0(AB)\sigma_{-2m}(D^{-2m})$:\\
Using the facts:
\begin{align}\label{a1}
&\Gamma^\mu(x_0)=\sigma^\mu(x_0)=\partial^x_\mu g^{\alpha\beta}(x_0)=0,~~~\partial_{x_k}(\Gamma^\mu)(x_0)=g^{\alpha\beta}\partial_{x_k}(\Gamma_{\alpha\beta}^\mu)(x_0)=\frac{2}{3}R_{k\alpha\mu\alpha}(x_0);\nonumber\\
&\partial_{x_k}(\sigma^\mu)(x_0)=\frac{1}{4}\sum_{st=1}^n\partial_{x_k}(<\nabla^L_{\partial_\mu }e_s,e_t>)(x_0)c(e_s)c(e_t)=\frac{1}{8}\sum_{st=1}^nR_{k\mu t s}(x_0)c(e_s)c(e_t).
\end{align}
In normal coordinates, we get the following result.
 \begin{align}
&\sigma_0(AB)\sigma_{-2m}(D^{-2m})(x_0)\nonumber\\
&=\sum_{kr}\partial_{x_k}^2(a^1)\partial_{x_r}^2(a^2)c(da^3)c(da^4)|\xi|^{-2m}+2\sum_{kp}\partial_{x_k}^2(a^1)\partial_{x_p}(a^2)\partial_{x_p}[c(da^3)c(da^4)]|\xi|^{-2m}\nonumber\\
&-\frac{4}{3}\sum_{jr\alpha}\partial_{x_j}(a^1)\partial_{x_r}(a^2)R_{j\alpha r\alpha}(x_0)c(da^3)c(da^4)|\xi|^{-2m}+2\sum_{jr}\partial_{x_j}(a^1)\partial_{x_r}^2(a^2)\partial_{x_j}[c(da^3)c(da^4)]|\xi|^{-2m}\nonumber\\
&+\frac{1}{2}\sum_{jrst}\partial_{x_j}(a^1)\partial_{x_r}(a^2)R_{jrts}(x_0)c(e_s)c(e_t)c(da^3)c(da^4)|\xi|^{-2m}+2\sum_{jr}\partial_{x_j}(a^1)\partial_{x_j}\partial_{x_r}^2(a^2)c(da^3)c(da^4)|\xi|^{-2m}\nonumber\\
&+2\sum_{jr}\partial_{x_j}(a^1)\partial_{x_j}\partial_{x_r}^2(a^2)c(da^3)c(da^4)|\xi|^{-2m}+4\sum_{jp}\partial_{x_j}(a^1)\partial_{x_j}\bigg(\partial_{x_p}(a^2)\partial_{x_p}[c(da^3)c(da^4)]\bigg)|\xi|^{-2m}.
\end{align}
$\mathbf{(i-1-a)}$
\begin{align*}
{\rm tr}\bigg(\sum_{kr}\partial_{x_k}^2(a^1)\partial_{x_r}^2(a^2)c(da^3)c(da^4)|\xi|^{-2m}\bigg)|_{|\xi|=1}&=\sum_{kr}\partial_{x_k}^2(a^1)\partial_{x_r}^2(a^2){\rm tr}[c(da^3)c(da^4)]\nonumber\\
&=-\sum_{kr}\partial_{x_k}^2(a^1)\partial_{x_r}^2(a^2)g(da^3,da^4){\rm tr}[id],\nonumber\\
\end{align*}
then
\begin{align*}
\int_{|\xi|=1}{\rm tr}\bigg(\sum_{kr}\partial_{x_k}^2(a^1)\partial_{x_r}^2(a^2)c(da^3)c(da^4)|\xi|^{-2m}\bigg)\sigma(\xi)&=-\sum_{kr}\partial_{x_k}^2(a^1)\partial_{x_r}^2(a^2)g(da^3,da^4){\rm tr}[id]area(S_n)\nonumber\\
&=-\Delta(a^1)\Delta(a^2)g(da^3,da^4){\rm tr}[id]area(S_n),
\end{align*}
where $\Delta(a^1)$ denotes a generalized laplacian of $a^1$.\\
$\mathbf{(i-1-b)}$
\begin{align*}
{\rm tr}\bigg(2\sum_{kp}\partial_{x_k}^2(a^1)\partial_{x_p}(a^2)\partial_{x_p}[c(da^3)c(da^4)]|\xi|^{-2m}\bigg)|_{|\xi|=1}&=-2\sum_{kp}\partial_{x_k}^2(a^1)\partial_{x_p}(a^2)\partial_{x_p}[g(da^3,da^4)]{\rm tr}[id],
\end{align*}
then
\begin{align*}
\int_{|\xi|=1}{\rm tr}\bigg(2\sum_{kp}\partial_{x_k}^2(a^1)\partial_{x_p}(a^2)\partial_{x_p}[c(da^3)c(da^4)]|\xi|^{-2m}\bigg)\sigma(\xi)&=-2\sum_{kp}\partial_{x_k}^2(a^1)\partial_{x_p}(a^2)\partial_{x_p}[g(da^3,da^4)]{\rm tr}[id]area(S_n)\nonumber\\
&=-2\Delta(a^1)g(\nabla (a^2),\nabla g(da^3,da^4)){\rm tr}[id]area(S_n),
\end{align*}
where $\nabla(a^2)$ denotes the gradient of $a^2$.\\
$\mathbf{(i-1-c)}$
\begin{align*}
{\rm tr}\bigg(-\frac{4}{3}\sum_{jr\alpha}\partial_{x_j}(a^1)\partial_{x_r}(a^2)R_{j\alpha r\alpha}(x_0)c(da^3)c(da^4)|\xi|^{-2m}\bigg)|_{|\xi|=1}&=\frac{4}{3}\sum_{jr\alpha}\partial_{x_j}(a^1)\partial_{x_r}(a^2)R_{j\alpha r\alpha}(x_0)g(da^3,da^4){\rm tr}[id],
\end{align*}
then
\begin{align*}
&\int_{|\xi|=1}{\rm tr}\bigg(-\frac{4}{3}\sum_{jr\alpha}\partial_{x_j}(a^1)\partial_{x_r}(a^2)R_{j\alpha r\alpha}(x_0)c(da^3)c(da^4)|\xi|^{-2m}\bigg)\sigma(\xi)\nonumber\\
&=\frac{4}{3}\sum_{jr\alpha}\partial_{x_j}(a^1)\partial_{x_r}(a^2)R_{j\alpha r\alpha}(x_0)g(da^3,da^4){\rm tr}[id]area(S_n)\nonumber\\
&=\frac{4}{3}\sum_\alpha R(\nabla (a^1),e_\alpha,\nabla (a^2),e_\alpha)g(da^3,da^4){\rm tr}[id]area(S_n).
\end{align*}
$\mathbf{(i-1-d)}$
\begin{align*}
{\rm tr}\bigg(2\sum_{jp}\partial_{x_j}(a^1)\partial_{x_p}^2(a^2)\partial_{x_j}[c(da^3)c(da^4)]|\xi|^{-2m}\bigg)|_{|\xi|=1}&=-2\sum_{jp}\partial_{x_j}(a^1)\partial_{x_p}^2(a^2)\partial_{x_j}[g(da^3,da^4)]{\rm tr}[id],
\end{align*}
then
\begin{align*}
&\int_{|\xi|=1}{\rm tr}\bigg(2\sum_{jp}\partial_{x_j}(a^1)\partial_{x_p}^2(a^2)\partial_{x_j}[c(da^3)c(da^4)]|\xi|^{-2m}\bigg)\sigma(\xi)\nonumber\\
&=-2\sum_{jp}\partial_{x_j}(a^1)\partial_{x_p}^2(a^2)\partial_{x_j}[g(da^3,da^4)]{\rm tr}[id]area(S_n)\nonumber\\
&=-2\Delta(a^2)g(\nabla (a^1),\nabla g(da^3,da^4)){\rm tr}[id]area(S_n),
\end{align*}
$\mathbf{(i-1-e)}$
\begin{align*}
&{\rm tr}\bigg(\sum_{jr}\partial_{x_j}(a^1)\partial_{x_r}(a^2)\sum_{st}R_{jrts}(x_0)c(e_s)c(e_t)c(da^3)c(da^4)|\xi|^{-2m}
\bigg)|_{|\xi|=1}\nonumber\\
&=\sum_{jr}\partial_{x_j}(a^1)\partial_{x_r}(a^2)\sum_{st}R_{jrts}(x_0){\rm tr}[c(e_s)c(e_t)c(da^3)c(da^4)],
\end{align*}
where
\begin{align*}
{\rm tr}[c(e_s)c(e_t)c(da^3)c(da^4)]&=\sum_{\alpha\beta}e_\alpha(a^3)e_\beta(a^4){\rm tr}[c(e_s)c(e_t)c(e_\alpha)c(e_\beta)]\nonumber\\
&=\sum_{\alpha\beta}e_\alpha(a^3)e_\beta(a^4)\bigg(-\delta_{st}{\rm tr}[c(e_\alpha)c(e_\beta)]+\delta_{s\alpha}{\rm tr}[c(e_t)c(e_\beta)]-\delta_{s\beta}{\rm tr}[c(e_t)c(e_\alpha)]\bigg),
\end{align*}
then
\begin{align*}
R_{jrts}(x_0){\rm tr}[c(e_s)c(e_t)c(da^3)c(da^4)]&=-R_{jrts}(x_0)e_s(a^3)e_t(a^4){\rm tr}[id]+R_{jrts}(x_0)e_t(a^3)e_s(a^4){\rm tr}[id]\nonumber\\
&=-2R_{jrts}(x_0)e_s(a^3)e_t(a^4){\rm tr}[id].
\end{align*}
Further,
\begin{align*}
&\int_{|\xi|=1}{\rm tr}\bigg(\frac{1}{2}\sum_{jr}\partial_{x_j}(a^1)\partial_{x_r}(a^2)\sum_{st}R_{jrts}(x_0)c(e_s)c(e_t)c(da^3)c(da^4)|\xi|^{-2m}\bigg)\sigma(\xi)\nonumber\\
&=-\sum_{jrts}\partial_{x_j}(a^1)\partial_{x_r}(a^2)R_{jrts}(x_0)e_s(a^3)e_t(a^4){\rm tr}[id]area(S_n)\nonumber\\
&=-R(\nabla(a^1),\nabla(a^2),\nabla(a^4),\nabla(a^3)){\rm tr}[id]area(S_n).
\end{align*}
$\mathbf{(i-1-f)}$(See Appendix)
\begin{align*}
&\int_{|\xi|=1}{\rm tr}\bigg(2\sum_{jr}\partial_{x_j}(a^1)\partial_{x_j}\partial_{x_r}^2(a^2)c(da^3)c(da^4)|\xi|^{-2m}\bigg)\sigma(\xi)\nonumber\\
&=-2\sum_{jr}\partial_{x_j}(a^1)\partial_{x_j}\partial_{x_r}^2(a^2)g(da^3,da^4){\rm tr}[id]area(S_n)\nonumber\\
&=\Big(2\nabla(a^1)[\Delta(a^2)]-\frac{4}{3}\sum_\alpha R(\nabla (a^1),e_\alpha,\nabla (a^2),e_\alpha)\Big)g(da^3,da^4){\rm tr}[id]area(S_n).
\end{align*}
$\mathbf{(i-1-g)}$(See Appendix)
\begin{align*}
\int_{|\xi|=1}{\rm tr}\bigg[4\sum_{jp}\partial_{x_j}(a^1)\partial_{x_j}\bigg(\partial_{x_p}(a^2)\partial_{x_p}[c(da^3)c(da^4)]\bigg)|\xi|^{-2m}\bigg]\sigma(\xi)=-4\nabla(a^1)\nabla(a^2)[g(da^3,da^4)]{\rm tr}[id]area(S_n).
\end{align*}
Therefore, we get
\begin{align*}
&\int_{|\xi|=1}{\rm tr}\bigg(\sigma_0(AB)\sigma_{-2m}(D^{-2m})(x_0)\bigg)\sigma(\xi)\nonumber\\
&=\bigg(-\Delta(a^1)\Delta(a^2)g(da^3,da^4)-2\Delta(a^1)g(\nabla (a^2),\nabla g(da^3,da^4))-R(\nabla(a^1),\nabla(a^2),\nabla(a^4),\nabla(a^3))\nonumber\\
&+2\nabla(a^1)[\Delta(a^2)]g(da^3,da^4)-4\nabla(a^1)\nabla(a^2)[g(da^3,da^4)] \bigg){\rm tr}[id]area(S_n).
\end{align*}
$\mathbf{(i-2)}$ For $\sigma_1(AB)\sigma_{-2m-1}(D^{-2m})$:

By (\ref{a1}), in normal coordinates, we get
\begin{align}\label{a33}
&\sigma_{-2m-1}(D^{-2m})(x_0)\nonumber\\
&=m|\xi|^{-(2m-2)}\bigg(-i|\xi|^{-4}\xi_k(\Gamma^k-2\sigma^k)(x_0)-2i|\xi|^{-6}\xi^j\xi_\alpha\xi_\beta\partial_{x_j}g^{\alpha\beta}(x_0)\bigg)\nonumber\\
&+2i\sum_{k=0}^{m-2}\sum_{\mu=1}^{2m}(k+1-m)|\xi|^{-2m-4}\xi^\mu\xi_\alpha\xi_\beta\partial^x_\mu g^{\alpha\beta}(x_0)\nonumber\\
&=0,
\end{align}
therefore
\begin{align*}
\int_{|\xi|=1}{\rm tr}\bigg(\sigma_1(AB)\sigma_{-2m-1}(D^{-2m})(x_0)\bigg)\sigma(\xi)=0.
\end{align*}
$\mathbf{(i-3)}$ For $\sigma_2(AB)\sigma_{-2m-2}(D^{-2m})$:
 \begin{align}\label{A1}
&\sigma_2(AB)\sigma_{-2m-2}(D^{-2m})(x_0)\nonumber\\
&=m\sum_{jl}\partial_{x_j}(a^1)\partial_{x_l}(a^2)\xi_j\xi_lc(da^3)c(da^4)|\xi|^{-2m-2}s\nonumber\\
&-\frac{4m(m+1)}{3}\sum_{jl}\partial_{x_j}(a^1)\partial_{x_l}(a^2)\sum_{\mu\nu\alpha}R_{\mu\alpha\nu\alpha}(x_0)\xi_j\xi_l\xi_\mu\xi_\nu c(da^3)c(da^4)|\xi|^{-2m-4}.
\end{align}
$\mathbf{(i-3-a)}$
\begin{align*}
m{\rm tr}\bigg(\sum_{jl}\partial_{x_j}(a^1)\partial_{x_l}(a^2)\xi_j\xi_lc(da^3)c(da^4)|\xi|^{-2m-2}s\bigg)|_{|\xi|=1}&=-m\sum_{jl}\partial_{x_j}(a^1)\partial_{x_l}(a^2)\xi_j\xi_lg(da^3,da^4)s{\rm tr}[id],
\end{align*}
then
\begin{align*}
&\int_{|\xi|=1}m{\rm tr}\bigg(\sum_{jl}\partial_{x_j}(a^1)\partial_{x_l}(a^2)\xi_j\xi_lc(da^3)c(da^4)|\xi|^{-2m-2}s\bigg)\sigma(\xi)\nonumber\\
&=-m\sum_{jl}\partial_{x_j}(a^1)\partial_{x_l}(a^2)g(da^3,da^4)s{\rm tr}[id]\int_{|\xi|=1}\xi_j\xi_l\sigma(\xi)\nonumber\\
&=-m\sum_{jl}\partial_{x_j}(a^1)\partial_{x_l}(a^2)g(da^3,da^4)s{\rm tr}[id]I^{jl}_{S_n}\nonumber\\
&=-\frac{1}{2}\sum_{j}\partial_{x_j}(a^1)\partial_{x_j}(a^2)area(S_n)g(da^3,da^4)s{\rm tr}[id]\nonumber\\
&=-\frac{1}{2}g(\nabla(a^1),\nabla(a^2))area(S_n)g(da^3,da^4)s{\rm tr}[id].
\end{align*}
$\mathbf{(i-3-b)}$
\begin{align*}
&{\rm tr}\bigg(-\frac{4m(m+1)}{3}\sum_{jl}\partial_{x_j}(a^1)\partial_{x_l}(a^2)\sum_{\mu\nu\alpha}R_{\mu\alpha\nu\alpha}(x_0)\xi_j\xi_l\xi_\mu\xi_\nu c(da^3)c(da^4)|\xi|^{-2m-4}\bigg)|_{|\xi|=1}\nonumber\\
&=\frac{4m(m+1)}{3}\sum_{jl}\partial_{x_j}(a^1)\partial_{x_l}(a^2)\sum_{\mu\nu\alpha}R_{\mu\alpha\nu\alpha}(x_0)\xi_j\xi_l\xi_\mu\xi_\nu g(da^3,da^4){\rm tr}[id],
\end{align*}
then
\begin{align*}
&\int_{|\xi|=1}-\frac{4m(m+1)}{3}{\rm tr}\bigg(\sum_{jl}\partial_{x_j}(a^1)\partial_{x_l}(a^2)\sum_{\mu\nu\alpha}R_{\mu\alpha\nu\alpha}(x_0)\xi_j\xi_l\xi_\mu\xi_\nu c(da^3)c(da^4)|\xi|^{-2m-4}\bigg)\sigma(\xi)\nonumber\\
&=\frac{4m(m+1)}{3}\sum_{jl}\partial_{x_j}(a^1)\partial_{x_l}(a^2)\sum_{\mu\nu\alpha}R_{\mu\alpha\nu\alpha}(x_0)g(da^3,da^4){\rm tr}[id]\int_{|\xi|=1}\xi_j\xi_l\xi_\mu\xi_\nu \sigma(\xi)\nonumber\\
&=\frac{4m(m+1)}{3}\sum_{jl}\partial_{x_j}(a^1)\partial_{x_l}(a^2)\sum_{\mu\nu\alpha}R_{\mu\alpha\nu\alpha}(x_0)g(da^3,da^4){\rm tr}[id]I_{S_n}^{jl\mu\nu}\nonumber\\
&=\frac{1}{3}\bigg(\sum_{j\mu\alpha}\partial_{x_j}(a^1)\partial_{x_j}(a^2)R_{\mu\alpha\mu\alpha}(x_0)+2\sum_{jl \alpha}\partial_{x_j}(a^1)\partial_{x_l}(a^2)R_{j\alpha l\alpha}(x_0)\bigg)area(S_n)g(da^3,da^4){\rm tr}[id]\nonumber\\
&=\frac{1}{3}\bigg(\sum_{\mu\alpha}g(\nabla(a^1),\nabla(a^2))R_{\mu\alpha\mu\alpha}(x_0)+2\sum_{\alpha }R(\nabla(a^1),e_\alpha,\nabla(a^2),e_\alpha)\bigg)area(S_n)g(da^3,da^4){\rm tr}[id]\nonumber\\
&=\frac{1}{3}\bigg(g(\nabla(a^1),\nabla(a^2))s+2\sum_{\alpha }R(\nabla(a^1),e_\alpha,\nabla(a^2),e_\alpha)\bigg)area(S_n)g(da^3,da^4){\rm tr}[id].
\end{align*}
Therefore, we get
\begin{align*}
&\int_{|\xi|=1}{\rm tr}\bigg(\sigma_2(AB)\sigma_{-2m-2}(D^{-2m})(x_0)\bigg)\sigma(\xi)\nonumber\\
&=\bigg(-\frac{1}{6}g(\nabla(a^1),\nabla(a^2))s+\frac{2}{3}\sum_{\alpha }R(\nabla(a^1),e_\alpha,\nabla(a^2),e_\alpha)\bigg)area(S_n)g(da^3,da^4){\rm tr}[id].
\end{align*}
$\mathbf{(i-4)}$ For $-i\sum_{j=1}^n\partial_{\xi_j}[\sigma_2(AB)]\partial_{x_j}[\sigma_{-2m-1}(D^{-2m})]$:
 \begin{align*}
&-i\sum_{j=1}^n\partial_{\xi_j}[\sigma_2(AB)]\partial_{x_j}[\sigma_{-2m-1}(D^{-2m})](x_0)\nonumber\\
&=\frac{8m}{3}\sum_{jplk\alpha}\partial_{x_p}(a^1)\partial_{x_l}(a^2)[\partial_{\xi_j}(\xi_p)\xi_l\xi_k+\partial_{\xi_j}(\xi_l)\xi_p\xi_k]R_{j\alpha k\alpha}(x_0)c(da^3)c(da^4)|\xi|^{-2m-2}\nonumber\\
&-m\sum_{jplkts}\partial_{x_p}(a^1)\partial_{x_l}(a^2)[\partial_{\xi_j}(\xi_p)\xi_l\xi_k+\partial_{\xi_j}(\xi_l)\xi_p\xi_k]R_{jkts}(x_0)c(da^3)c(da^4)c(e_s)c(e_t)|\xi|^{-2m-2}\nonumber\\
&+\frac{16}{3}m\sum_{jplr\alpha\beta}\partial_{x_p}(a^1)\partial_{x_l}(a^2)R_{\alpha j\beta r}(x_0)[\partial_{\xi_j}(\xi_p)\xi_l\xi_r\xi_\alpha\xi_\beta+\partial_{\xi_j}(\xi_l)\xi_p\xi_r\xi_\alpha\xi_\beta]c(da^3)c(da^4)|\xi|^{-2m-4}\nonumber\\
&+\frac{16(m-k-1)}{3}\sum_{k=0}^{m-2}\sum_{\mu=1}^{2m}\sum_{jplr\alpha\beta}\partial_{x_p}(a^1)\partial_{x_l}(a^2)R_{\alpha j\beta \mu}(x_0)[\partial_{\xi_j}(\xi_p)\xi_l\xi_\mu\xi_\alpha\xi_\beta+\partial_{\xi_j}(\xi_l)\xi_p\xi_\mu\xi_\alpha\xi_\beta]c(da^3)c(da^4)|\xi|^{-2m-4}.
\end{align*}
$\mathbf{(i-4-a)}$
\begin{align*}
&{\rm tr}\bigg(\frac{8}{3}m\sum_{jplk\alpha}\partial_{x_p}(a^1)\partial_{x_l}(a^2)[\partial_{\xi_j}(\xi_p)\xi_l\xi_k+\partial_{\xi_j}(\xi_l)\xi_p\xi_k]R_{j\alpha k\alpha}(x_0)c(da^3)c(da^4)|\xi|^{-2m-2}\bigg)|_{|\xi|=1}\nonumber\\
&=-\frac{8}{3}m\sum_{jplk\alpha}\partial_{x_p}(a^1)\partial_{x_l}(a^2)[\partial_{\xi_j}(\xi_p)\xi_l\xi_k+\partial_{\xi_j}(\xi_l)\xi_p\xi_k]R_{j\alpha k\alpha}(x_0)g(da^3,da^4){\rm tr}[id],
\end{align*}
then
\begin{align*}
&\int_{|\xi|=1}{\rm tr}\bigg(\frac{8}{3}m\sum_{jplk\alpha}\partial_{x_p}(a^1)\partial_{x_l}(a^2)[\partial_{\xi_j}(\xi_p)\xi_l\xi_k+\partial_{\xi_j}(\xi_l)\xi_p\xi_k]R_{j\alpha k\alpha}(x_0)c(da^3)c(da^4)|\xi|^{-2m-2}\bigg)\sigma(\xi)\nonumber\\
&=-\frac{8}{3}m\sum_{jplk\alpha}\partial_{x_p}(a^1)\partial_{x_l}(a^2)R_{j\alpha k\alpha}(x_0)g(a^3,a^4){\rm tr}[id]\int_{|\xi|=1}[\partial_{\xi_j}(\xi_p)\xi_l\xi_k+\partial_{\xi_j}(\xi_l)\xi_p\xi_k]\sigma(\xi)\nonumber\\
&=-\frac{8}{3}\sum_{plk\alpha}\partial_{x_p}(a^1)\partial_{x_l}(a^2)R_{p\alpha k\alpha}(x_0) area(S_n)g(da^3,da^4){\rm tr}[id]\nonumber\\
&=-\frac{8}{3}\sum_{pk\alpha}\partial_{x_p}(a^1)\partial_{x_k}(a^2)R_{p\alpha k\alpha}(x_0) area(S_n)g(da^3,da^4){\rm tr}[id]\nonumber\\
&=-\frac{8}{3}\sum_{\alpha}R(\nabla(a^1),e_\alpha,\nabla(a^2),e_\alpha) area(S_n)g(da^3,da^4){\rm tr}[id].
\end{align*}
$\mathbf{(i-4-b)}$
\begin{align*}
&{\rm tr}\bigg(-m\sum_{jplkts}\partial_{x_p}(a^1)\partial_{x_l}(a^2)[\partial_{\xi_j}(\xi_p)\xi_l\xi_k+\partial_{\xi_j}(\xi_l)\xi_p\xi_k]R_{jkts}(x_0)c(da^3)c(da^4)c(e_s)c(e_t)|\xi|^{-2m-2}\bigg)|_{|\xi|=1}\nonumber\\
&=2m\sum_{jplkts}\partial_{x_p}(a^1)\partial_{x_l}(a^2)[\partial_{\xi_j}(\xi_p)\xi_l\xi_k+\partial_{\xi_j}(\xi_l)\xi_p\xi_k]R_{jkts}(x_0)e_s(da^3)e_t(da^4){\rm tr}[id],
\end{align*}
then
\begin{align*}
&\int_{|\xi|=1}{\rm tr}\bigg(-m\sum_{jplkts}\partial_{x_p}(a^1)\partial_{x_l}(a^2)[\partial_{\xi_j}(\xi_p)\xi_l\xi_k+\partial_{\xi_j}(\xi_l)\xi_p\xi_k]R_{jkts}(x_0)c(da^3)c(da^4)c(e_s)c(e_t)|\xi|^{-2m-2}\bigg)\sigma(\xi)\nonumber\\
&=area(S^n)\sum_{jlts}\partial_{x_j}(a^1)\partial_{x_l}(a^2)R_{jlts}(x_0)e_s(a^3)e_t(a^4){\rm tr}[id]\nonumber\\
&+area(S^n)\sum_{jpts}\partial_{x_p}(a^1)\partial_{x_j}(a^2)R_{jpts}(x_0)e_s(a^3)e_t(a^4){\rm tr}[id]\nonumber\\
&=area(S^n)\Big(R(\nabla(a^1),\nabla(a^2),\nabla(a^4),\nabla(a^3))+R(\nabla(a^2),\nabla(a^1),\nabla(a^4),\nabla(a^3))\Big){\rm tr}[id]\nonumber\\
&=0.
\end{align*}
$\mathbf{(i-4-c)}$
\begin{align*}
&{\rm tr}\bigg(\sum_{jplr\alpha\beta}\partial_{x_p}(a^1)\partial_{x_l}(a^2)R_{\alpha j\beta r}(x_0)[\partial_{\xi_j}(\xi_p)\xi_l\xi_r\xi_\alpha\xi_\beta+\partial_{\xi_j}(\xi_l)\xi_p\xi_r\xi_\alpha\xi_\beta]c(da^3)c(da^4)|\xi|^{-2m-4}\bigg)|_{|\xi|=1}\nonumber\\
&=-\sum_{jplr\alpha\beta}\partial_{x_p}(a^1)\partial_{x_l}(a^2)R_{\alpha j\beta r}(x_0)[\partial_{\xi_j}(\xi_p)\xi_l\xi_r\xi_\alpha\xi_\beta+\partial_{\xi_j}(\xi_l)\xi_p\xi_r\xi_\alpha\xi_\beta]c(da^3)c(da^4)|\xi|^{-2m-4}g(da^3,da^4){\rm tr}[id],
\end{align*}
then
\begin{align*}
&\int_{|\xi|=1}{\rm tr}\bigg(\sum_{jplr\alpha\beta}\partial_{x_p}(a^1)\partial_{x_l}(a^2)R_{\alpha j\beta r}(x_0)[\partial_{\xi_j}(\xi_p)\xi_l\xi_r\xi_\alpha\xi_\beta+\partial_{\xi_j}(\xi_l)\xi_p\xi_r\xi_\alpha\xi_\beta]c(da^3)c(da^4)|\xi|^{-2m-4}\bigg)\sigma(\xi)\nonumber\\
&=-\sum_{jplr\alpha\beta}\partial_{x_p}(a^1)\partial_{x_l}(a^2)R_{\alpha j\beta r}(x_0)g(da^3,da^4){\rm tr}[id]\int_{|\xi|=1}[\partial_{\xi_j}(\xi_p)\xi_l\xi_r\xi_\alpha\xi_\beta+\partial_{\xi_j}(\xi_l)\xi_p\xi_r\xi_\alpha\xi_\beta]\sigma(\xi)\nonumber\\
&=-\bigg(\sum_{jlr\alpha\beta}\partial_{x_j}(a^1)\partial_{x_l}(a^2)R_{\alpha j\beta r}(x_0)I_{S_n}^{lr\alpha\beta}+\sum_{jr\alpha\beta}\partial_{x_p}(a^1)\partial_{x_j}(a^2)R_{\alpha j\beta r}(x_0)I_{S_n}^{pr\alpha\beta}\bigg)g(da^3,da^4){\rm tr}[id]\nonumber\\
&=0.
\end{align*}
$\mathbf{(i-4-d)}$\\
Similarly, we get
\begin{align*}
&\int_{|\xi|=1}{\rm tr}[\sum_{jplr\alpha\beta}\partial_{x_p}(a^1)\partial_{x_l}(a^2)R_{\alpha j\beta \mu}(x_0)[\partial_{\xi_j}(\xi_p)\xi_l\xi_\mu\xi_\alpha\xi_\beta+\partial_{\xi_j}(\xi_l)\xi_p\xi_\mu\xi_\alpha\xi_\beta]c(da^3)c(da^4)|\xi|^{-2m-4}]\sigma(\xi)=0.
\end{align*}
Therefore, we get
\begin{align*}
&\int_{|\xi|=1}{\rm tr}\bigg(-i\sum_{j=1}^n\partial_{\xi_j}[\sigma_2(AB)]\partial_{x_j}[\sigma_{-2m-1}(D^{-2m})](x_0)\bigg)\sigma(\xi)\nonumber\\
&=-\frac{8}{3}\sum_{\alpha}R(\nabla(a^1),e_\alpha,\nabla(a^2),e_\alpha) area(S_n)g(da^3,da^4){\rm tr}[id].
\end{align*}
$\mathbf{(i-5)}$ For $-i\Sigma_{j=1}^n\partial_{\xi_j}[\sigma_1(AB)]\partial_{x_j}[\sigma_{-2m}(D^{-2m})]$:
\begin{align}\label{a44}
\partial_{x_j}(|\xi|^{-2m})(x_0)\sigma(\xi)=0,
\end{align}
then
\begin{align*}
\int_{|\xi|=1}{\rm tr}\bigg(-i\Sigma_{j=1}^n\partial_{\xi_j}[\sigma_1(AB)]\partial_{x_j}[\sigma_{-2m}(D^{-2m})]\bigg)(x_0)\sigma(\xi)=0.
\end{align*}
$\mathbf{(i-6)}$ For $-\frac{1}{2}\sum_{jl}\partial_{\xi_j}\partial_{\xi_l}[\sigma_2(AB)]\partial_{x_j}\partial_{x_l}[\sigma_{-2m}(D^{-2m})]$:
 \begin{align*}
&-\frac{1}{2}\sum_{jl}\partial_{\xi_j}\partial_{\xi_l}[\sigma_2(AB)]\partial_{x_j}\partial_{x_l}[\sigma_{-2m}(D^{-2m})](x_0)\nonumber\\
&=-\frac{4}{3}m\sum_{jlrs\alpha\beta}\partial_{x_r}(a^1)\partial_{x_s}(a^2)R_{\alpha j\beta l}(x_0)[\partial_{\xi_l}(\xi_r)\partial_{\xi_j}(\xi_s)+\partial_{\xi_l}(\xi_s)\partial_{\xi_j}(\xi_r)]\xi_\alpha\xi_\beta c(da^3)c(da^4)|\xi|^{-2m-2},
\end{align*}
so
\begin{align*}
&{\rm tr}\bigg(-\frac{4}{3}m\sum_{jlrs\alpha\beta}\partial_{x_r}(a^1)\partial_{x_s}(a^2)R_{\alpha j\beta l}(x_0)[\partial_{\xi_l}(\xi_r)\partial_{\xi_j}(\xi_s)+\partial_{\xi_l}(\xi_s)\partial_{\xi_j}(\xi_r)]\xi_\alpha\xi_\beta c(da^3)c(da^4)|\xi|^{-2m-2}\bigg)|_{|\xi|=1}\nonumber\\
&=\frac{4}{3}m\sum_{jlrs\alpha\beta}\partial_{x_r}(a^1)\partial_{x_s}(a^2)R_{\alpha j\beta l}(x_0)[\partial_{\xi_l}(\xi_r)\partial_{\xi_j}(\xi_s)+\partial_{\xi_l}(\xi_s)\partial_{\xi_j}(\xi_r)]\xi_\alpha\xi_\beta g(da^3,da^4){\rm tr}[id],
\end{align*}
then
\begin{align*}
&\int_{|\xi|=1}-\frac{4}{3}m{\rm tr}\bigg(\sum_{jlrs\alpha\beta}\partial_{x_r}(a^1)\partial_{x_s}(a^2)R_{\alpha j\beta l}(x_0)[\partial_{\xi_l}(\xi_r)\partial_{\xi_j}(\xi_s)+\partial_{\xi_l}(\xi_s)\partial_{\xi_j}(\xi_r)]\xi_\alpha\xi_\beta c(da^3)c(da^4)|\xi|^{-2m-2}\bigg)\sigma(\xi)\nonumber\\
&=\frac{4}{3}m\sum_{jlrs\alpha\beta}\partial_{x_r}(a^1)\partial_{x_s}(a^2)[\partial_{\xi_l}(\xi_r)\partial_{\xi_j}(\xi_s)+\partial_{\xi_l}(\xi_s)\partial_{\xi_j}(\xi_r)]R_{\alpha j\beta l}(x_0) g(da^3,da^4){\rm tr}[id]\int_{|\xi|=1}\xi_\alpha\xi_\beta\sigma(\xi)\nonumber\\
&=\frac{4}{3}m\sum_{jlrs\alpha\beta}\partial_{x_r}(a^1)\partial_{x_s}(a^2)[\partial_{\xi_l}(\xi_r)\partial_{\xi_j}(\xi_s)+\partial_{\xi_l}(\xi_s)\partial_{\xi_j}(\xi_r)]R_{\alpha j\beta l}(x_0) g(da^3,da^4){\rm tr}[id]I_{S_n}^{\alpha\beta}\nonumber\\
&=\frac{4}{3}\sum_{jl\alpha}\partial_{x_j}(a^1)\partial_{x_l}(a^2)R_{\alpha j\alpha l}(x_0)area(S_n)g(da^3,da^4){\rm tr}[id]\nonumber\\
&=\frac{4}{3}\sum_{\alpha}R(\nabla(a^1),e_\alpha,\nabla(a^2),e_\alpha)area(S_n)g(da^3,da^4){\rm tr}[id].
\end{align*}
Therefore, we get
\begin{align*}
&\int_{|\xi|=1}{\rm tr}\bigg(-\frac{1}{2}\sum_{jl}\partial_{\xi_j}\partial_{\xi_l}[\sigma_2(AB)]\partial_{x_j}\partial_{x_l}[\sigma_{-2m}(D^{-2m})](x_0)\bigg)\sigma(\xi)\nonumber\\
&=\frac{4}{3}\sum_{\alpha}R(\nabla(a^1),e_\alpha,\nabla(a^2),e_\alpha)area(S_n)g(da^3,da^4){\rm tr}[id].
\end{align*}
Finally, we obtain
\begin{thm} Let $P=[D^2,a^1][D^2,a^2][D,a^3][D,a^4],$ we have the following noncommutative integral of $P$ based on the Wodzicki residue
\begin{align*}
&\int\hspace{-1.05em}-a^0[D^2,a^1][D^2,a^2][D,a^3][D,a^4]D^{-2m}\nonumber\\
&=2^m\frac{2\pi^{\frac{n}{2}}}{\Gamma(\frac{n}{2})}\int_Ma_0\bigg(-\Delta(a^1)\Delta(a^2)g(da^3,da^4)-2\Delta(a^1)g(\nabla (a^2),\nabla g(da^3,da^4))-R(\nabla(a^1),\nabla(a^2),\nabla(a^4),\nabla(a^3))\nonumber\\
&+2\nabla(a^1)[\Delta(a^2)]g(da^3,da^4)-4\nabla(a^1)\nabla(a^2)[g(da^3,da^4)]-\frac{1}{6}g(\nabla(a^1),\nabla(a^2))g(da^3,da^4)s\nonumber\\
&-\frac{2}{3}\sum_{\alpha}R(\nabla(a^1),e_\alpha,\nabla(a^2),e_\alpha)g(da^3,da^4)
\bigg)d{\rm Vol_M}.
\end{align*}
\end{thm}
$\mathbf{Part~~ii)}$ $P_1=[D,a^1][D^2,a^2][D^2,a^3][D,a^4]$, $[D,a^1][D^2,a^2]:=A_1,~[D^2,a^3][D,a^4]:=B_1$.\\
Similarly, by Lemma 3.1 in \cite{UW} and Lemma \ref{lemma1}, we get the following lemma.
\begin{lem}The symbols of $A_1$ and $B_1$ are given
\begin{align}
&\sigma_0(A_1)=c(da^1)\bigg(\sum_{j=1}^n\partial_{x_j}(a^2)(\Gamma^j-2\sigma^j)(x)-\sum_{jl=1}^n\partial_{x_j}\partial_{x_l}(a^2)g^{jl}\bigg)\nonumber;\\
&\sigma_1(A_1)=-2ic(da^1)\sum_{jl=1}^n\partial_{x_j}(a^2)g^{jl}\xi_l\nonumber;\\
&\sigma_0(B_1)=\bigg(\sum_{j=1}^n\partial_{x_j}(a^3)(\Gamma^j-2\sigma^j)(x)-\sum_{jl=1}^n\partial_{x_j}\partial_{x_l}(a^3)g^{jl}\bigg)c(da^4)-2\sum_{jl=1}^n\partial_{x_l}(a^3)g^{jl}\partial_{x_j}[c(da^4)]\nonumber;\\
&\sigma_1(B_1)=-2i\sum_{jl=1}^n\partial_{x_j}(a^3)g^{jl}\xi_lc(da^4).
\end{align}
\end{lem}
Next, we calculate each term of  $\int_{S^*M}{\rm tr}[\sigma_{-2m}(A_1B_1D^{-2m})](x,\xi)$ separately.\\
$\mathbf{(ii-1)}$ For $\sigma_0(A_1B_1)\sigma_{-2m}(D^{-2m})$:
 \begin{align}
&\sigma_0(A_1B_1)\sigma_{-2m}(D^{-2m})(x_0)\nonumber\\
&=\sum_{kr}\partial_{x_k}^2(a^2)\partial_{x_r}^2(a^3)c(da^1)c(da^4)|\xi|^{-2m}+2\sum_{kp}\partial_{x_k}^2(a^2)\partial_{x_p}(a^3)c(da^1)\partial_{x_p}[c(da^4)]|\xi|^{-2m}\nonumber\\
&-\frac{4}{3}\sum_{jr\alpha}\partial_{x_j}(a^2)\partial_{x_r}(a^3)R_{j\alpha r\alpha}(x_0)c(da^1)c(da^4)|\xi|^{-2m}+2\sum_{jp}\partial_{x_j}(a^2)\partial_{x_p}^2(a^3)c(da^1)\partial_{x_j}[c(da^4)]|\xi|^{-2m}\nonumber\\
&+\frac{1}{2}\sum_{jrst}\partial_{x_j}(a^2)\partial_{x_r}(a^3)R_{jrts}(x_0)c(e_s)c(e_t)c(da^1)c(da^4)|\xi|^{-2m}+2\sum_{jr}\partial_{x_j}(a^2)\partial_{x_j}\partial_{x_r}^2(a^3)c(da^1)c(da^4)|\xi|^{-2m}\nonumber\\
&+4\sum_{jp}\partial_{x_j}(a^2)\partial_{x_p}(a^3)c(da^1)\partial_{x_j}\partial_{x_p}[c(da^4)]|\xi|^{-2m}++4\sum_{jp}\partial_{x_j}(a^2)\partial_{x_j}\partial_{x_p}(a^3)c(da^1)\partial_{x_p}[c(da^4)]|\xi|^{-2m}.
\end{align}
$\mathbf{(ii-1-a)}$
\begin{align*}
\int_{|\xi|=1}{\rm tr}\bigg(\sum_{kr}\partial_{x_k}^2(a^2)\partial_{x_r}^2(a^3)c(da^1)c(da^4)|\xi|^{-2m}\bigg)\sigma(\xi)&=-\sum_{kr}\partial_{x_k}^2(a^2)\partial_{x_r}^2(a^3)g(da^1,da^4){\rm tr}[id]\nonumber\\
&=-\Delta(a^2)\Delta(a^3)g(da^1,da^4){\rm tr}[id]area(S_n) .
\end{align*}
$\mathbf{(ii-1-b)}$(See Appendix)
\begin{align*}
&\int_{|\xi|=1}{\rm tr}\bigg(2\sum_{kp}\partial_{x_k}^2(a^2)\partial_{x_p}(a^3)c(da^1)\partial_{x_p}[c(da^4)]|\xi|^{-2m}\bigg)\sigma(\xi)\nonumber\\
&=-2\sum_{kp}\partial_{x_k}^2(a^2)\partial_{x_p}(a^3)\sum_\mu e_\mu(a^1)\partial_{x_p}[e_\mu(a^4)]{\rm tr}[id]area(S_n) \nonumber\\
&=2\Delta(a^2)g(\nabla(a^1),\nabla_{\nabla(a^3)}\nabla(a^4)){\rm tr}[id]area(S_n) .
\end{align*}
$\mathbf{(ii-1-c)}$
\begin{align*}
&\int_{|\xi|=1}{\rm tr}\bigg(-\frac{4}{3}\sum_{jr\alpha}\partial_{x_j}(a^2)\partial_{x_r}(a^3)R_{j\alpha r\alpha}(x_0)c(da^1)c(da^4)|\xi|^{-2m}\bigg)\sigma(\xi)\nonumber\\
&=\frac{4}{3}\sum_{jr\alpha}\partial_{x_j}(a^2)\partial_{x_r}(a^3)R_{j\alpha r\alpha}(x_0)g(da^1,da^4){\rm tr}[id]area(S_n) \nonumber\\
&=\frac{4}{3}\sum_\alpha R(\nabla(a^2),e_\alpha,\nabla(a^3),e_\alpha)g(da^1,da^4){\rm tr}[id]area(S_n).
\end{align*}
$\mathbf{(ii-1-d)}$
\begin{align*}
&\int_{|\xi|=1}{\rm tr}\bigg(2\sum_{jp}\partial_{x_j}(a^2)\partial_{x_p}^2(a^3)c(da^1)\partial_{x_j}[c(da^4)]|\xi|^{-2m}\bigg)\sigma(\xi)\nonumber\\
&=-2\sum_{jp}\partial_{x_j}(a^2)\partial_{x_p}^2(a^3)\sum_\mu e_\mu(a^1)\partial_{x_j}[e_\mu(a^4)]{\rm tr}[id]area(S_n) \nonumber\\
&=2\Delta(a^3)g(\nabla(a^1),\nabla_{\nabla(a^2)}\nabla(a^4)){\rm tr}[id]area(S_n) .
\end{align*}
$\mathbf{(ii-1-e)}$
\begin{align*}
&\int_{|\xi|=1}{\rm tr}\bigg(\frac{1}{2}\sum_{jr}\partial_{x_j}(a^2)\partial_{x_r}(a^3)\sum_{st}R_{jrts}(x_0)c(e_s)c(e_t)c(da^1)c(da^4)|\xi|^{-2m}\bigg)\sigma(\xi)\nonumber\\
&=-\sum_{jrst}\partial_{x_j}(a^2)\partial_{x_r}(a^3)R_{jrts}(x_0)e_s(a^1)e_t(a^4){\rm tr}[id]area(S_n) \nonumber\\
&=-R(\nabla(a^2),\nabla(a^3),\nabla(a^4),\nabla(a^1)){\rm tr}[id]area(S_n) .
\end{align*}
$\mathbf{(ii-1-f)}$
\begin{align*}
&\int_{|\xi|=1}{\rm tr}\bigg(2\sum_{jr}\partial_{x_j}(a^2)\partial_{x_j}\partial_{x_r}^2(a^3)c(da^1)c(da^4)|\xi|^{-2m}\bigg)\sigma(\xi)\nonumber\\
&=-2\sum_{jr}\partial_{x_j}(a^2)\partial_{x_j}\partial_{x_r}^2(a^3)g(da^1,da^4){\rm tr}[id]area(S_n) \nonumber\\
&=\bigg(2\nabla(a^2)[\Delta(a^3)]-\frac{4}{3}\sum_\alpha R(\nabla (a^2),e_\alpha,\nabla (a^3),e_\alpha)\bigg)g(da^1,da^4){\rm tr}[id]area(S_n) .
\end{align*}
$\mathbf{(ii-1-g)}$(See Appendix)
\begin{align*}
&\int_{|\xi|=1}{\rm tr}\bigg(4\sum_{jp}\partial_{x_j}(a^2)\partial_{x_p}(a^3)c(da^1)\partial_{x_j}\partial_{x_p}[c(da^4)]|\xi|^{-2m}\bigg)\sigma(\xi)\nonumber\\
&=-4\sum_{jp}\partial_{x_j}(a^2)\partial_{x_p}(a^3)\sum_\mu e_\mu(a^1)\partial_{x_j}\partial_{x_p}[e_\mu(a^4)]{\rm tr}[id]area(S_n) \nonumber\\
&=-4\Big(g(\nabla(a^1),\nabla_{\nabla(a^2)}\nabla_{\nabla(a^3)}\nabla(a^4))-g(\nabla_{\nabla(a^1)}\nabla(a^4),\nabla_{\nabla(a^2)}\nabla(a^3))\nonumber\\
&-\frac{1}{2}R(\nabla(a^1),\nabla(a^3),\nabla(a^2),\nabla(a^4))\Big){\rm tr}[id]area(S_n) .
\end{align*}
$\mathbf{(ii-1-h)}$(See Appendix)
\begin{align*}
&\int_{|\xi|=1}{\rm tr}\bigg(4\sum_{jp}\partial_{x_j}(a^2)\partial_{x_j}\partial_{x_p}(a^3)c(da^1)\partial_{x_p}[c(da^4)]|\xi|^{-2m}\bigg)\sigma(\xi)\nonumber\\
&=-4\sum_{jp}\partial_{x_j}(a^2)\partial_{x_j}\partial_{x_p}(a^3)\sum_\mu e_\mu(a^1)\partial_{x_p}[e_\mu(a^4)]{\rm tr}[id]area(S_n) \nonumber\\
&=4g(\nabla(a^1),\nabla_{\nabla_{\nabla(a^2)}\nabla(a^3)}\nabla(a^4)){\rm tr}[id]area(S_n) .
\end{align*}
Therefore, we get
\begin{align*}
&\int_{|\xi|=1}{\rm tr}\bigg(\sigma_0(A_1B_1)\sigma_{-2m}(D^{-2m})(x_0)\bigg)\sigma(\xi)\nonumber\\
&=\bigg(-\Delta(a^2)\Delta(a^3)g(da^1,da^4)+2\Delta(a^2)g(\nabla(a^1),\nabla_{\nabla(a^3)}\nabla(a^4))-R(\nabla(a^2),\nabla(a^3),\nabla(a^4),\nabla(a^1))\nonumber\\
&+2\nabla(a^2)[\Delta(a^3)]g(da^1,da^4)+2\Delta(a^3)g(\nabla(a^1),\nabla_{\nabla(a^2)}\nabla(a^4))-4g(\nabla(a^1),\nabla_{\nabla(a^2)}\nabla_{\nabla(a^3)}\nabla(a^4))\nonumber\\
&+4g(\nabla_{\nabla(a^1)}\nabla(a^4),\nabla_{\nabla(a^2)}\nabla(a^3))+2R(\nabla(a^1),\nabla(a^3),\nabla(a^2),\nabla(a^4))+4g(\nabla(a^1),\nabla_{\nabla_{\nabla(a^2)}\nabla(a^3)}\nabla(a^4))\bigg){\rm tr}[id]area(S_n) .
\end{align*}
$\mathbf{(ii-2)}$ For $\sigma_1(A_1B_1)\sigma_{-2m-1}(D^{-2m})$:\\
By (\ref{a33}), we obtain
\begin{align*}
\int_{|\xi|=1}{\rm tr}\bigg(\sigma_1(A_1B_1)\sigma_{-2m-1}(D^{-2m})(x_0)\bigg)\sigma(\xi)=0.
\end{align*}
$\mathbf{(3)}$ For $\sigma_2(A_1B_1)\sigma_{-2m-2}(D^{-2m})$:\\
By (\ref{a1}), in normal coordinates, we get the following result.
 \begin{align}\label{A1}
&\sigma_2(A_1B_1)\sigma_{-2m-2}(D^{-2m})(x_0)\nonumber\\
&=m\sum_{jl}\partial_{x_j}(a^2)\partial_{x_l}(a^3)\xi_j\xi_lc(da^1)(da^4)|\xi|^{-2m-2}s\nonumber\\
&-\frac{4m(m+1)}{3}\sum_{jl}\partial_{x_j}(a^2)\partial_{x_l}(a^3)\sum_{\mu\nu\alpha}R_{\mu\alpha\nu\alpha}(x_0)\xi_j\xi_l\xi_\mu\xi_\nu c(da^1)c(da^4)|\xi|^{-2m-4}.
\end{align}
$\mathbf{(ii-3-a)}$
\begin{align*}
&\int_{|\xi|=1}m{\rm tr}\bigg(\sum_{jl}\partial_{x_j}(a^2)\partial_{x_l}(a^3)\xi_j\xi_lc(da^1)c(da^4)|\xi|^{-2m-2}s\bigg)\sigma(\xi)\nonumber\\
&=-\frac{1}{2}\sum_{j}\partial_{x_j}(a^2)\partial_{x_j}(a^3)area(S_n)g(da^1,da^4)s{\rm tr}[id]\nonumber\\
&=-\frac{1}{2}g(\nabla(a^2),\nabla(a^3))area(S_n)g(da^1,da^4)s{\rm tr}[id].
\end{align*}
$\mathbf{(ii-3-b)}$
\begin{align*}
&\int_{|\xi|=1}\frac{-4m(m+1)}{3}{\rm tr}\bigg(\sum_{jl}\partial_{x_j}(a^2)\partial_{x_l}(a^3)\sum_{\mu\nu\alpha}R_{\mu\alpha\nu\alpha}(x_0)\xi_j\xi_l\xi_\mu\xi_\nu c(da^1)c(da^4)|\xi|^{-2m-4}\bigg)\sigma(\xi)\nonumber\\
&=\frac{1}{3}\bigg(\sum_{j\mu\alpha}\partial_{x_j}(a^2)\partial_{x_j}(a^3)R_{\mu\alpha\mu\alpha}(x_0)+2\sum_{jl \alpha}\partial_{x_j}(a^2)\partial_{x_l}(a^3)R_{j\alpha l\alpha}(x_0)\bigg)area(S_n)g(da^1,da^4){\rm tr}[id]\nonumber\\
&=\frac{1}{3}\bigg(\sum_{\mu\alpha}g(\nabla(a^2),\nabla(a^3))R_{\mu\alpha\mu\alpha}(x_0)+2\sum_{\alpha }R(\nabla(a^2),e_\alpha,\nabla(a^3),e_\alpha)\bigg)area(S_n)g(da^1,da^4){\rm tr}[id]\nonumber\\
&=\frac{1}{3}\bigg(g(\nabla(a^2),\nabla(a^3))s+2\sum_{\alpha }R(\nabla(a^2),e_\alpha,\nabla(a^3),e_\alpha)\bigg)area(S_n)g(da^1,da^4){\rm tr}[id].
\end{align*}
Therefore, we get
\begin{align*}
&\int_{|\xi|=1}{\rm tr}\bigg(\sigma_2(A_1B_1)\sigma_{-2m-2}(D^{-2m})(x_0)\bigg)\sigma(\xi)\nonumber\\
&=\bigg(-\frac{1}{6}g(\nabla(a^2),\nabla(a^3))s+\frac{2}{3}\sum_{\alpha }R(\nabla(a^2),e_\alpha,\nabla(a^3),e_\alpha)\bigg)area(S_n)g(da^1,da^4){\rm tr}[id].
\end{align*}
$\mathbf{(ii-4)}$ For $-i\sum_{j=1}^n\partial_{\xi_j}[\sigma_2(A_1B_1)]\partial_{x_j}[\sigma_{-2m-1}(D^{-2m})]$:
 \begin{align*}
&-i\sum_{j=1}^n\partial_{\xi_j}[\sigma_2(A_1B_1)]\partial_{x_j}[\sigma_{-2m-1}(D^{-2m})](x_0)\nonumber\\
&=\frac{8m}{3}\sum_{jplk\alpha}\partial_{x_p}(a^2)\partial_{x_l}(a^3)[\partial_{\xi_j}(\xi_p)\xi_l\xi_k+\partial_{\xi_j}(\xi_l)\xi_p\xi_k]R_{j\alpha k\alpha}(x_0)c(da^1)c(da^4)|\xi|^{-2m-2}\nonumber\\
&-m\sum_{jplkts}\partial_{x_p}(a^2)\partial_{x_l}(a^3)[\partial_{\xi_j}(\xi_p)\xi_l\xi_k+\partial_{\xi_j}(\xi_l)\xi_p\xi_k]R_{jkts}(x_0)c(da^1)c(da^4)c(e_s)c(e_t)|\xi|^{-2m-2}\nonumber\\
&+\frac{16}{3}m\sum_{jplr\alpha\beta}\partial_{x_p}(a^2)\partial_{x_l}(a^3)R_{\alpha j\beta r}(x_0)[\partial_{\xi_j}(\xi_p)\xi_l\xi_r\xi_\alpha\xi_\beta+\partial_{\xi_j}(\xi_l)\xi_p\xi_r\xi_\alpha\xi_\beta]c(da^1)c(da^4)|\xi|^{-2m-4}\nonumber\\
&+\frac{16(m-k-1)}{3}\sum_{k=0}^{m-2}\sum_{\mu=1}^{2m}\sum_{jplr\alpha\beta}\partial_{x_p}(a^2)\partial_{x_l}(a^3)R_{\alpha j\beta \mu}(x_0)[\partial_{\xi_j}(\xi_p)\xi_l\xi_\mu\xi_\alpha\xi_\beta+\partial_{\xi_j}(\xi_l)\xi_p\xi_\mu\xi_\alpha\xi_\beta]c(da^1)c(da^4)|\xi|^{-2m-4}.
\end{align*}
$\mathbf{(ii-4-a)}$
\begin{align*}
&\int_{|\xi|=1}{\rm tr}\bigg(\frac{8m}{3}\sum_{jplk\alpha}\partial_{x_p}(a^2)\partial_{x_l}(a^3)[\partial_{\xi_j}(\xi_p)\xi_l\xi_k+\partial_{\xi_j}(\xi_l)\xi_p\xi_k]R_{j\alpha k\alpha}(x_0)c(da^1)c(da^4)|\xi|^{-2m-2}\bigg)\sigma(\xi)\nonumber\\
&=-\frac{8}{3}\sum_{pk\alpha}\partial_{x_p}(a^2)\partial_{x_k}(a^3)R_{p\alpha k\alpha}(x_0) area(S_n)g(da^1,da^4){\rm tr}[id]\nonumber\\
&=-\frac{8}{3}\sum_{\alpha}R(\nabla(a^2),e_\alpha,\nabla(a^3),e_\alpha) area(S_n)g(da^1,da^4){\rm tr}[id].
\end{align*}
$\mathbf{(ii-4-b)}$
\begin{align*}
&\int_{|\xi|=1}{\rm tr}\bigg(-m\sum_{jplkts}\partial_{x_p}(a^2)\partial_{x_l}(a^3)[\partial_{\xi_j}(\xi_p)\xi_l\xi_k+\partial_{\xi_j}(\xi_l)\xi_p\xi_k]R_{jkts}(x_0)c(da^1)c(da^4)c(e_s)c(e_t)|\xi|^{-2m-2}\bigg)\sigma(\xi)\nonumber\\
&=area(S^n)\sum_{jlts}\partial_{x_j}(a^2)\partial_{x_l}(a^3)R_{jlts}(x_0)e_s(a^1)e_t(a^4){\rm tr}[id]\nonumber\\
&+area(S^n)\sum_{jpts}\partial_{x_p}(a^2)\partial_{x_j}(a^3)R_{jpts}(x_0)e_s(a^1)e_t(a^4){\rm tr}[id]\nonumber\\
&=area(S^n)\Big(R(\nabla(a^2),\nabla(a^3),\nabla(a^4),\nabla(a^1))+R(\nabla(a^3),\nabla(a^2),\nabla(a^4),\nabla(a^1))\Big){\rm tr}[id]\nonumber\\
&=0.
\end{align*}
$\mathbf{(ii-4-c)}$
\begin{align*}
&\int_{|\xi|=1}{\rm tr}\bigg(\sum_{jplr\alpha\beta}\partial_{x_p}(a^2)\partial_{x_l}(a^3)R_{\alpha j\beta r}(x_0)[\partial_{\xi_j}(\xi_p)\xi_l\xi_r\xi_\alpha\xi_\beta+\partial_{\xi_j}(\xi_l)\xi_p\xi_r\xi_\alpha\xi_\beta]c(da^1)c(da^4)|\xi|^{-2m-4}\bigg)\sigma(\xi)\nonumber\\
&=-\bigg(\sum_{jlr\alpha\beta}\partial_{x_j}(a^2)\partial_{x_l}(a^3)R_{\alpha j\beta r}(x_0)I_{S_n}^{lr\alpha\beta}+\sum_{jr\alpha\beta}\partial_{x_p}(a^2)\partial_{x_j}(a^3)R_{\alpha j\beta r}(x_0)I_{S_n}^{pr\alpha\beta}\bigg)g(da^1,da^4){\rm tr}[id]\nonumber\\
&=0.
\end{align*}
$\mathbf{(ii-4-d)}$\\
Similarly, we get
\begin{align*}
&\int_{|\xi|=1}{\rm tr}\bigg(\sum_{jplr\alpha\beta}\partial_{x_p}(a^2)\partial_{x_l}(a^3)R_{\alpha j\beta \mu}(x_0)[\partial_{\xi_j}(\xi_p)\xi_l\xi_\mu\xi_\alpha\xi_\beta+\partial_{\xi_j}(\xi_l)\xi_p\xi_\mu\xi_\alpha\xi_\beta]c(da^1)c(da^4)|\xi|^{-2m-4}\bigg)\sigma(\xi)=0.
\end{align*}
Therefore, we get
\begin{align*}
&\int_{|\xi|=1}{\rm tr}\bigg(-i\sum_{j=1}^n\partial_{\xi_j}[\sigma_2(A_1B_1)]\partial_{x_j}[\sigma_{-2m-1}(D^{-2m})](x_0)\bigg)\sigma(\xi)\nonumber\\
&=-\frac{8}{3}\sum_{\alpha}R(\nabla(a^2),e_\alpha,\nabla(a^3),e_\alpha) area(S_n)g(da^1,da^4){\rm tr}[id].
\end{align*}
$\mathbf{(ii-5)}$ For $-i\Sigma_{j=1}^n\partial_{\xi_j}[\sigma_1(A_1B_1)]\partial_{x_j}[\sigma_{-2m}(D^{-2m})]$:\\
By (\ref{a44}), we get
\begin{align*}
\int_{|\xi|=1}{\rm tr}\bigg(-i\Sigma_{j=1}^n\partial_{\xi_j}[\sigma_1(A_1B_1)]\partial_{x_j}[\sigma_{-2m}(D^{-2m})]\bigg)(x_0)\sigma(\xi)=0.
\end{align*}
$\mathbf{(ii-6)}$ For $-\frac{1}{2}\sum_{jl}\partial_{\xi_j}\partial_{\xi_l}[\sigma_2(A_1B_1)]\partial_{x_j}\partial_{x_l}[\sigma_{-2m}(D^{-2m})]$:
\begin{align*}
&\int_{|\xi|=1}-\frac{4}{3}m{\rm tr}\bigg(\sum_{jlrs\alpha\beta}\partial_{x_r}(a^2)\partial_{x_s}(a^3)R_{\alpha j\beta l}(x_0)[\partial_{\xi_l}(\xi_r)\partial_{\xi_j}(\xi_s)+\partial_{\xi_l}(\xi_s)\partial_{\xi_j}(\xi_r)]\xi_\alpha\xi_\beta c(da^1)c(da^4)|\xi|^{-2m-2}\bigg)\sigma(\xi)\nonumber\\
&=\frac{4}{3}\sum_{jl\alpha}\partial_{x_j}(a^2)\partial_{x_l}(a^3)R_{\alpha j\alpha l}(x_0)area(S_n)g(da^1,da^4){\rm tr}[id]\nonumber\\
&=\frac{4}{3}\sum_{\alpha}R(\nabla(a^2),e_\alpha,\nabla(a^3),e_\alpha) area(S_n)g(da^1,da^4){\rm tr}[id].
\end{align*}
Therefore, we get
\begin{align*}
&\int_{|\xi|=1}{\rm tr}\bigg(-\frac{1}{2}\sum_{jl}\partial_{\xi_j}\partial_{\xi_l}[\sigma_2(A_1B_1)]\partial_{x_j}\partial_{x_l}[\sigma_{-2m}(D^{-2m})](x_0)\bigg)\sigma(\xi)\nonumber\\
&=\frac{4}{3}\sum_{\alpha}R(\nabla(a^2),e_\alpha,\nabla(a^3),e_\alpha)area(S_n)g(da^1,da^4){\rm tr}[id].
\end{align*}
Finally, we obtain
\begin{thm}
Let $P_1=[D,a^1][D^2,a^2][D^2,a^3][D,a^4]$, we have the following noncommutative integral of $P_1$ based on the Wodzicki residue
\begin{align*}
&\int\hspace{-1.05em}-a^0[D,a^1][D^2,a^2][D^2,a^3][D,a^4]D^{-2m}\nonumber\\
&=2^m\frac{2\pi^{\frac{n}{2}}}{\Gamma(\frac{n}{2})}\int_Ma_0 \bigg(-\Delta(a^2)\Delta(a^3)g(da^1,da^4)+2\Delta(a^2)g(\nabla(a^1),\nabla_{\nabla(a^3)}\nabla(a^4))-R(\nabla(a^2),\nabla(a^3),\nabla(a^4),\nabla(a^1))\nonumber\\
&+2\nabla(a^2)[\Delta(a^3)]g(da^1,da^4)+2\Delta(a^3)g(\nabla(a^1),\nabla_{\nabla(a^2)}\nabla(a^4))-4g(\nabla(a^1),\nabla_{\nabla(a^2)}\nabla_{\nabla(a^3)}\nabla(a^4))\nonumber\\
&+4g(\nabla_{\nabla(a^1)}\nabla(a^4),\nabla_{\nabla(a^2)}\nabla(a^3))+2R(\nabla(a^1),\nabla(a^3),\nabla(a^2),\nabla(a^4))+4g(\nabla(a^1),\nabla_{\nabla_{\nabla(a^2)}\nabla(a^3)}\nabla(a^4))\nonumber\\
&-\frac{1}{6}g(\nabla(a^2),\nabla(a^3))g(da^1,da^4)s-\frac{2}{3}\sum_{\alpha }R(\nabla(a^2),e_\alpha,\nabla(a^3),e_\alpha)g(da^1,da^4)
\bigg)d{\rm Vol_M}.
\end{align*}
\end{thm}
$\mathbf{Part~~iii)}$ $P_2=[D,a^1][D,a^2][D^2,a^3][D^2,a^4]$, $[D,a^1][D,a^2][D^2,a^3]:=A_2,~[D^2,a^4]:=B_2$.\\
Similarly, by Lemma 3.1 in \cite{UW} and Lemma \ref{lemma1}, we get the following lemma.
\begin{lem}The symbols of $A$ and $B$ are given
\begin{align}
&\sigma_0(A_2)=c(da^1)c(da^2)[\sum_{j=1}^n\partial_{x_j}(a^3)(\Gamma^j-2\sigma^j)(x)-\sum_{jl=1}^n\partial_{x_j}\partial_{x_l}(a^3)g^{jl}]\nonumber;\\
&\sigma_1(A_2)=-2ic(da^1)c(da^2)\sum_{jl=1}^n\partial_{x_j}(a^3)g^{jl}\xi_l\nonumber;\\
&\sigma_0(B_2)=\sum_{j=1}^n\partial_{x_j}(a^4)(\Gamma^j-2\sigma^j)(x)-\sum_{jl=1}^n\partial_{x_j}\partial_{x_l}(a^4)g^{jl}\nonumber;\\
&\sigma_1(B_2)=-2i\sum_{jl=1}^n\partial_{x_j}(a^4)g^{jl}\xi_l.
\end{align}
\end{lem}
Next, we calculate each term of  $\int_{S^*M}{\rm tr}[\sigma_{-2m}(A_2B_2D^{-2m})](x,\xi)$ separately.\\
$\mathbf{(iii-1)}$ For $\sigma_0(A_2B_2)\sigma_{-2m}(D^{-2m})$
 \begin{align}
&\sigma_0(A_2B_2)\sigma_{-2m}(D^{-2m})(x_0)\nonumber\\
&=\sum_{kr}\partial_{x_k}^2(a^3)\partial_{x_r}^2(a^4)c(da^1)c(da^2)|\xi|^{-2m}-\frac{4}{3}\sum_{jr\alpha}\partial_{x_j}(a^3)\partial_{x_r}(a^4)R_{j\alpha r\alpha}(x_0)c(da^1)c(da^2)|\xi|^{-2m}\nonumber\\
&+\frac{1}{2}\sum_{jrst}\partial_{x_j}(a^3)\partial_{x_r}(a^4)R_{jrts}(x_0)c(e_s)c(e_t)c(da^1)c(da^2)|\xi|^{-2m}\nonumber\\
&+2\sum_{jr}\partial_{x_j}(a^3)\partial_{x_j}\partial_{x_r}^2(a^4)c(da^1)c(da^2)|\xi|^{-2m}\bigg)|\xi|^{-2m}.
\end{align}
$\mathbf{(iii-1-a)}$
\begin{align*}
\int_{|\xi|=1}{\rm tr}\bigg(\sum_{kr}\partial_{x_k}^2(a^3)\partial_{x_r}^2(a^4)c(da^1)c(da^2)|\xi|^{-2m}\bigg)\sigma(\xi)&=-\sum_{kr}\partial_{x_k}^2(a^3)\partial_{x_r}^2(a^4)g(da^1,da^2){\rm tr}[id]area(S_n)\nonumber\\
&=-\Delta(a^3)\Delta(a^4)g(da^1,da^2){\rm tr}[id]area(S_n).
\end{align*}
$\mathbf{(iii-1-b)}$
\begin{align*}
&\int_{|\xi|=1}{\rm tr}\bigg(-\frac{4}{3}\sum_{jr\alpha}\partial_{x_j}(a^3)\partial_{x_r}(a^4)R_{j\alpha r\alpha}(x_0)c(da^1)c(da^2)|\xi|^{-2m}\bigg)\sigma(\xi)\nonumber\\
&=\frac{4}{3}\sum_{jr\alpha}\partial_{x_j}(a^3)\partial_{x_r}(a^4)R_{j\alpha r\alpha}(x_0)g(da^1,da^2){\rm tr}[id]area(S_n)\nonumber\\
&=\frac{4}{3}\sum_\alpha R(\nabla (a^3),e_\alpha,\nabla (a^4),e_\alpha)g(da^1,da^2){\rm tr}[id]area(S_n).
\end{align*}
$\mathbf{(iii-1-c)}$
\begin{align*}
&\int_{|\xi|=1}{\rm tr}\bigg(\frac{1}{2}\sum_{jr}\partial_{x_j}(a^3)\partial_{x_r}(a^4)\sum_{st}R_{jrts}(x_0)c(e_s)c(e_t)c(da^1)c(da^2)|\xi|^{-2m}\bigg)\sigma(\xi)\nonumber\\
&=-\sum_{jrst}\partial_{x_j}(a^3)\partial_{x_r}(a^4)R_{jrts}(x_0)e_s(a^1)e_t(a^2){\rm tr}[id]area(S_n)\nonumber\\
&=-R(\nabla(a^3),\nabla(a^4),\nabla(a^2),\nabla(a^1)){\rm tr}[id]area(S_n).
\end{align*}
$\mathbf{(iii-1-d)}$
\begin{align*}
&\int_{|\xi|=1}{\rm tr}\bigg(2\sum_{jr}\partial_{x_j}(a^3)\partial_{x_j}\partial_{x_r}^2(a^4)c(da^1)c(da^2)|\xi|^{-2m}\bigg)\sigma(\xi)\nonumber\\
&=-2\sum_{jr}\partial_{x_j}(a^3)\partial_{x_j}\partial_{x_r}^2(a^4)g(da^1,da^2){\rm tr}[id]area(S_n)\nonumber\\
&=\bigg(2\nabla(a^3)[\Delta(a^4)]-\frac{4}{3}\sum_\alpha R(\nabla (a^3),e_\alpha,\nabla (a^4),e_\alpha)\bigg)g(da^1,da^2){\rm tr}[id]area(S_n).
\end{align*}
Therefore, we get
\begin{align*}
&\int_{|\xi|=1}{\rm tr}\bigg(\sigma_0(A_2B_2)\sigma_{-2m}(D^{-2m})(x_0)\bigg)\sigma(\xi)\nonumber\\
&=\Big(-\Delta(a^3)\Delta(a^4)-R(\nabla(a^3),\nabla(a^4),\nabla(a^2),\nabla(a^1))+2\nabla(a^3)[\Delta(a^4)]\Big)g(da^1,da^2){\rm tr}[id]area(S_n).
\end{align*}
$\mathbf{(iii-2)}$ For $\sigma_1(A_2B_2)\sigma_{-2m-1}(D^{-2m})$:\\
By (\ref{a33}), we get
\begin{align*}
\int_{|\xi|=1}{\rm tr}\bigg(\sigma_1(A_2B_2)\sigma_{-2m-1}(D^{-2m})(x_0)\bigg)\sigma(\xi)=0.
\end{align*}
$\mathbf{(iii-3)}$ For $\sigma_2(A_2B_2)\sigma_{-2m-2}(D^{-2m})$:\\
 \begin{align}\label{A1}
\sigma_2(A_2B_2)\sigma_{-2m-2}(D^{-2m})(x_0)
&=m\sum_{jl}\partial_{x_j}(a^3)\partial_{x_l}(a^4)\xi_j\xi_lc(da^1)c(da^2)|\xi|^{-2m-2}s\nonumber\\
&-\frac{4m(m+1)}{3}\sum_{jl}\partial_{x_j}(a^3)\partial_{x_l}(a^4)\sum_{\mu\nu\alpha}R_{\mu\alpha\nu\alpha}(x_0)\xi_j\xi_l\xi_\mu\xi_\nu c(da^1)c(da^2)|\xi|^{-2m-4}.
\end{align}
$\mathbf{(iii-3-a)}$
\begin{align*}
&\int_{|\xi|=1}m{\rm tr}\bigg(\sum_{jl}\partial_{x_j}(a^3)\partial_{x_l}(a^4)\xi_j\xi_lc(da^1)c(da^2)|\xi|^{-2m-2}s\bigg)\sigma(\xi)\nonumber\\
&=-\frac{1}{2}\sum_{j}\partial_{x_j}(a^3)\partial_{x_j}(a^4)area(S_n)g(da^1,da^2)s{\rm tr}[id]\nonumber\\
&=-\frac{1}{2}g(\nabla(a^3),\nabla(a^4))area(S_n)g(da^1,da^2)s{\rm tr}[id].
\end{align*}
$\mathbf{(iii-3-b)}$
\begin{align*}
&\int_{|\xi|=1}-\frac{4m(m+1)}{3}{\rm tr}\bigg(\sum_{jl}\partial_{x_j}(a^3)\partial_{x_l}(a^4)\sum_{\mu\nu\alpha}R_{\mu\alpha\nu\alpha}(x_0)\xi_j\xi_l\xi_\mu\xi_\nu c(da^1)c(da^2)|\xi|^{-2m-4}\bigg)\sigma(\xi)\nonumber\\
&=\frac{1}{3}\bigg(\sum_{j\mu\alpha}\partial_{x_j}(a^3)\partial_{x_j}(a^4)R_{\mu\alpha\mu\alpha}(x_0)+2\sum_{jl \alpha}\partial_{x_j}(a^3)\partial_{x_l}(a^4)R_{j\alpha l\alpha}(x_0)\bigg)area(S_n)g(da^1,da^2){\rm tr}[id]\nonumber\\
&=\frac{1}{3}\bigg(\sum_{\mu\alpha}g(\nabla(a^3),\nabla(a^4))R_{\mu\alpha\mu\alpha}(x_0)+2\sum_{\alpha }R(\nabla(a^3),e_\alpha,\nabla(a^4),e_\alpha)\bigg)area(S_n)g(da^1,da^2){\rm tr}[id]\nonumber\\
&=\frac{1}{3}\bigg(g(\nabla(a^3),\nabla(a^4))s+2\sum_{\alpha }R(\nabla(a^3),e_\alpha,\nabla(a^4),e_\alpha)\bigg)area(S_n)g(da^1,da^2){\rm tr}[id].
\end{align*}
Therefore, we get
\begin{align*}
&\int_{|\xi|=1}{\rm tr}\bigg(\sigma_2(A_2B_2)\sigma_{-2m-2}(D^{-2m})(x_0)\bigg)\sigma(\xi)\nonumber\\
&=\bigg(-\frac{1}{6}g(\nabla(a^3),\nabla(a^4))s+\frac{2}{3}\sum_{\alpha }R(\nabla(a^3),e_\alpha,\nabla(a^4),e_\alpha)\bigg)area(S_n)g(da^1,da^2){\rm tr}[id].
\end{align*}
$\mathbf{(iii-4)}$ For $-i\sum_{j=1}^n\partial_{\xi_j}[\sigma_2(A_2B_2)]\partial_{x_j}[\sigma_{-2m-1}(D^{-2m})]$:
 \begin{align*}
&-i\sum_{j=1}^n\partial_{\xi_j}[\sigma_2(A_2B_2)]\partial_{x_j}[\sigma_{-2m-1}(D^{-2m})](x_0)\nonumber\\
&=\frac{8m}{3}\sum_{jplk\alpha}\partial_{x_p}(a^3)\partial_{x_l}(a^4)[\partial_{\xi_j}(\xi_p)\xi_l\xi_k+\partial_{\xi_j}(\xi_l)\xi_p\xi_k]R_{j\alpha k\alpha}(x_0)c(da^1)c(da^2)|\xi|^{-2m-2}\nonumber\\
&-m\sum_{jplkts}\partial_{x_p}(a^3)\partial_{x_l}(a^4)[\partial_{\xi_j}(\xi_p)\xi_l\xi_k+\partial_{\xi_j}(\xi_l)\xi_p\xi_k]R_{jkts}(x_0)c(da^1)c(da^2)c(e_s)c(e_t)|\xi|^{-2m-2}\nonumber\\
&+\frac{16}{3}m\sum_{jplr\alpha\beta}\partial_{x_p}(a^3)\partial_{x_l}(a^4)R_{\alpha j\beta r}(x_0)[\partial_{\xi_j}(\xi_p)\xi_l\xi_r\xi_\alpha\xi_\beta+\partial_{\xi_j}(\xi_l)\xi_p\xi_r\xi_\alpha\xi_\beta]c(da^1)c(da^2)|\xi|^{-2m-4}\nonumber\\
&+\frac{16(m-k-1)}{3}\sum_{k=0}^{m-2}\sum_{\mu=1}^{2m}\sum_{jplr\alpha\beta}\partial_{x_p}(a^3)\partial_{x_l}(a^4)R_{\alpha j\beta \mu}(x_0)[\partial_{\xi_j}(\xi_p)\xi_l\xi_\mu\xi_\alpha\xi_\beta+\partial_{\xi_j}(\xi_l)\xi_p\xi_\mu\xi_\alpha\xi_\beta]c(da^1)c(da^2)|\xi|^{-2m-4}.
\end{align*}
$\mathbf{(iii-4-a)}$
\begin{align*}
&\int_{|\xi|=1}{\rm tr}\bigg(\frac{8m}{3}\sum_{jplk\alpha}\partial_{x_p}(a^3)\partial_{x_l}(a^4)[\partial_{\xi_j}(\xi_p)\xi_l\xi_k+\partial_{\xi_j}(\xi_l)\xi_p\xi_k]R_{j\alpha k\alpha}(x_0)c(da^1)c(da^2)|\xi|^{-2m-2}\bigg)\sigma(\xi)\nonumber\\
&=-\frac{8}{3}\sum_{pk\alpha}\partial_{x_p}(a^3)\partial_{x_k}(a^4)R_{p\alpha k\alpha}(x_0) area(S_n)g(da^1,da^2){\rm tr}[id]\nonumber\\
&=-\frac{8}{3}\sum_{\alpha}R(\nabla(a^3),e_\alpha,\nabla(a^4),e_\alpha) area(S_n)g(da^1,da^2){\rm tr}[id].
\end{align*}
$\mathbf{(iii-4-b)}$
\begin{align*}
&\int_{|\xi|=1}{\rm tr}\bigg(-m\sum_{jplkts}\partial_{x_p}(a^3)\partial_{x_l}(a^4)[\partial_{\xi_j}(\xi_p)\xi_l\xi_k+\partial_{\xi_j}(\xi_l)\xi_p\xi_k]R_{jkts}(x_0)c(da^1)c(da^2)c(e_s)c(e_t)|\xi|^{-2m-2}\bigg)\sigma(\xi)\nonumber\\
&=area(S^n)\sum_{jlts}\partial_{x_j}(a^3)\partial_{x_l}(a^4)R_{jlts}(x_0)e_s(a^1)e_t(a^2){\rm tr}[id]\nonumber\\
&+area(S^n)\sum_{jpts}\partial_{x_p}(a^3)\partial_{x_j}(a^4)R_{jpts}(x_0)e_s(a^1)e_t(a^2){\rm tr}[id]\nonumber\\
&=area(S^n)\Big(R(\nabla(a^3),\nabla(a^4),\nabla(a^2),\nabla(a^1))+R(\nabla(a^4),\nabla(a^3),\nabla(a^2),\nabla(a^1))\Big){\rm tr}[id]\nonumber\\
&=0..
\end{align*}
$\mathbf{(iii-4-c)}$
\begin{align*}
&\int_{|\xi|=1}{\rm tr}\bigg(\sum_{jplr\alpha\beta}\partial_{x_p}(a^3)\partial_{x_l}(a^4)R_{\alpha j\beta r}(x_0)[\partial_{\xi_j}(\xi_p)\xi_l\xi_r\xi_\alpha\xi_\beta+\partial_{\xi_j}(\xi_l)\xi_p\xi_r\xi_\alpha\xi_\beta]c(da^1)c(da^2)|\xi|^{-2m-4}\bigg)\sigma(\xi)\nonumber\\
&=-\bigg(\sum_{jlr\alpha\beta}\partial_{x_j}(a^3)\partial_{x_l}(a^4)R_{\alpha j\beta r}(x_0)I_{S_n}^{lr\alpha\beta}+\sum_{jr\alpha\beta}\partial_{x_p}(a^3)\partial_{x_j}(a^4)R_{\alpha j\beta r}(x_0)I_{S_n}^{pr\alpha\beta}\bigg)g(da^1,da^2){\rm tr}[id]\nonumber\\
&=0.
\end{align*}
$\mathbf{(iii-4-d)}$\\
Similarly, we get
\begin{align*}
&\int_{|\xi|=1}{\rm tr}\bigg(\sum_{jplr\alpha\beta}\partial_{x_p}(a^3)\partial_{x_l}(a^4)R_{\alpha j\beta \mu}(x_0)[\partial_{\xi_j}(\xi_p)\xi_l\xi_\mu\xi_\alpha\xi_\beta+\partial_{\xi_j}(\xi_l)\xi_p\xi_\mu\xi_\alpha\xi_\beta]c(da^1)c(da^2)|\xi|^{-2m-4}\bigg)\sigma(\xi)=0.
\end{align*}
Therefore, we get
\begin{align*}
&\int_{|\xi|=1}{\rm tr}\bigg(-i\sum_{j=1}^n\partial_{\xi_j}[\sigma_2(A_2B_2)]\partial_{x_j}[\sigma_{-2m-1}(D^{-2m})](x_0)\bigg)\sigma(\xi)\nonumber\\
&=-\frac{8}{3}\sum_{\alpha}R(\nabla(a^3),e_\alpha,\nabla(a^4),e_\alpha) area(S_n)g(da^1,da^2){\rm tr}[id].
\end{align*}
$\mathbf{(iii-5)}$ For $-i\Sigma_{j=1}^n\partial_{\xi_j}[\sigma_1(A_2B_2)]\partial_{x_j}[\sigma_{-2m}(D^{-2m})]$:\\
In normal coordinates, by (\ref{a44}), we get
\begin{align*}
\int_{|\xi|=1}{\rm tr}\bigg(-i\Sigma_{j=1}^n\partial_{\xi_j}[\sigma_1(A_2B_2)]\partial_{x_j}[\sigma_{-2m}(D^{-2m})]\bigg)(x_0)\sigma(\xi)=0.
\end{align*}
$\mathbf{(iii-6)}$ For $-\frac{1}{2}\sum_{jl}\partial_{\xi_j}\partial_{\xi_l}[\sigma_2(A_2B_2)]\partial_{x_j}\partial_{x_l}[\sigma_{-2m}(D^{-2m})]$:\\
By (\ref{a1}) then in normal coordinates, we get the following result.
\begin{align*}
&\int_{|\xi|=1}-\frac{4}{3}m{\rm tr}\bigg(\sum_{jlrs\alpha\beta}\partial_{x_r}(a^3)\partial_{x_s}(a^4)R_{\alpha j\beta l}(x_0)[\partial_{\xi_l}(\xi_r)\partial_{\xi_j}(\xi_s)+\partial_{\xi_l}(\xi_s)\partial_{\xi_j}(\xi_r)]\xi_\alpha\xi_\beta c(da^1)c(da^2)|\xi|^{-2m-2}\bigg)\sigma(\xi)\nonumber\\
&=\frac{4}{3}\sum_{jl\alpha}\partial_{x_j}(a^3)\partial_{x_l}(a^4)R_{\alpha j\alpha l}(x_0)area(S_n)g(da^1,da^2){\rm tr}[id]\nonumber\\
&=\frac{4}{3}\sum_{\alpha}R(\nabla(a^3),e_\alpha,\nabla(a^4),e_\alpha) area(S_n)g(da^1,da^2){\rm tr}[id].
\end{align*}
Therefore, we get
\begin{align*}
&\int_{|\xi|=1}{\rm tr}\bigg(-\frac{1}{2}\sum_{jl}\partial_{\xi_j}\partial_{\xi_l}[\sigma_2(A_2B_2)]\partial_{x_j}\partial_{x_l}[\sigma_{-2m}(D^{-2m})](x_0)\bigg)\sigma(\xi)\nonumber\\
&=\frac{4}{3}\sum_{\alpha}R(\nabla(a^3),e_\alpha,\nabla(a^4),e_\alpha) area(S_n)g(da^1,da^2){\rm tr}[id].
\end{align*}
Finally, we obtain
\begin{thm}
Let $P_2=[D,a^1][D,a^2][D^2,a^3][D^2,a^4]$, we have the following noncommutative integral of $P_2$ based on the Wodzicki residue
\begin{align*}
&\int\hspace{-1.05em}-a^0[D,a^1][D,a^2][D^2,a^3][D^2,a^4]D^{-2m}\nonumber\\
&=2^m\frac{2\pi^{\frac{n}{2}}}{\Gamma(\frac{n}{2})} \int_Ma_0\bigg[\bigg(-\Delta(a^3)\Delta(a^4)+2\nabla(a^3)[\Delta(a^4)]-R(\nabla(a^3),\nabla(a^4),\nabla(a^2),\nabla(a^1))\nonumber\\
&-\frac{1}{6}g(\nabla(a^3),\nabla(a^4))s-\frac{2}{3}\sum_{\alpha }R(\nabla(a^3),e_\alpha,\nabla(a^4),e_\alpha)(x_0)\bigg)g(da^1,da^2)\bigg]d{\rm Vol_M}.
\end{align*}
\end{thm}
$\mathbf{Part~~iv)}$ $P_3=[D^2,a^1][D,a^2][D,a^3][D^2,a^4]$, $[D^2,a^1][D,a^2][D,a^3]:=A_3,~[D^2,a^4]:=B_3$.\\
Similarly, by Lemma 3.1 in \cite{UW} and Lemma \ref{lemma1}, we get the following lemma.
\begin{lem}The symbols of $A$ and $B$ are given
\begin{align}
&\sigma_0(A_3)=[\sum_{j=1}^n\partial_{x_j}(a^1)(\Gamma^j-2\sigma^j)(x)-\sum_{jl=1}^n\partial_{x_j}\partial_{x_l}(a^1)g^{jl}]c(da^2)c(da^3)\nonumber;\\
&\sigma_1(A_3)=-2i\sum_{jl=1}^n\partial_{x_j}(a^1)g^{jl}\xi_lc(da^2)c(da^3)\nonumber;\\
&\sigma_0(B_3)=\sum_{j=1}^n\partial_{x_j}(a^4)(\Gamma^j-2\sigma^j)(x)-\sum_{jl=1}^n\partial_{x_j}\partial_{x_l}(a^4)g^{jl}\nonumber;\\
&\sigma_1(B_3)=-2i\sum_{jl=1}^n\partial_{x_j}(a^4)g^{jl}\xi_l.
\end{align}
\end{lem}
Next, we calculate each term of  $\int_{S^*M}{\rm tr}[\sigma_{-2m}(A_3B_3D^{-2m})](x,\xi)$ separately.\\
$\mathbf{(iv-1)}$ For $\sigma_0(A_3B_3)\sigma_{-2m}(D^{-2m})$
 \begin{align}
&\sigma_0(A_3B_3)\sigma_{-2m}(D^{-2m})(x_0)\nonumber\\
&=\sum_{kr}\partial_{x_k}^2(a^1)\partial_{x_r}^2(a^4)c(da^2)c(da^3)|\xi|^{-2m}+2\sum_{kp}\partial_{x_k}^2(a^1)\partial_{x_p}(a^4)\partial_{x_p}[c(da^2)c(da^3)]|\xi|^{-2m}\nonumber\\
&-\frac{4}{3}\sum_{jr\alpha}\partial_{x_j}(a^1)\partial_{x_r}(a^4)R_{j\alpha r\alpha}(x_0)c(da^2)c(da^3)|\xi|^{-2m}+\frac{1}{2}\sum_{jrst}\partial_{x_j}(a^1)\partial_{x_r}(a^4)R_{jrts}(x_0)c(e_s)c(e_t)c(da^2)c(da^3)|\xi|^{-2m}\nonumber\\
&+2\sum_{jr}\partial_{x_j}(a^1)\partial_{x_j}\partial_{x_r}^2(a^4)c(da^2)c(da^3)|\xi|^{-2m}.
\end{align}
$\mathbf{(iv-1-a)}$
\begin{align*}
\int_{|\xi|=1}{\rm tr}\bigg(\sum_{kr}\partial_{x_k}^2(a^1)\partial_{x_r}^2(a^4)c(da^2)c(da^3)|\xi|^{-2m}\bigg)\sigma(\xi)&=-\sum_{kr}\partial_{x_k}^2(a^1)\partial_{x_r}^2(a^4)g(da^2,da^3){\rm tr}[id]area(S_n)\nonumber\\
&=-\Delta(a^1)\Delta(a^4)g(da^2,da^3){\rm tr}[id]area(S_n).
\end{align*}
$\mathbf{(iv-1-b)}$
\begin{align*}
\int_{|\xi|=1}{\rm tr}\bigg(2\sum_{kp}\partial_{x_k}^2(a^1)\partial_{x_p}(a^4)\partial_{x_p}[c(da^2)c(da^3)]|\xi|^{-2m}\bigg)\sigma(\xi)&=-2\sum_{kp}\partial_{x_k}^2(a^1)\partial_{x_p}(a^4)\partial_{x_p}[g(da^2,da^3)]{\rm tr}[id]area(S_n)\nonumber\\
&=-2\Delta(a^1)g(\nabla (a^4),\nabla g(da^2,da^3)){\rm tr}[id]area(S_n).
\end{align*}
$\mathbf{(iv-1-c)}$
\begin{align*}
&\int_{|\xi|=1}{\rm tr}\bigg(-\frac{4}{3}\sum_{jr\alpha}\partial_{x_j}(a^1)\partial_{x_r}(a^4)R_{j\alpha r\alpha}(x_0)c(da^2)c(da^3)|\xi|^{-2m}\bigg)\sigma(\xi)\nonumber\\
&=\frac{4}{3}\sum_{jr\alpha}\partial_{x_j}(a^1)\partial_{x_r}(a^4)R_{j\alpha r\alpha}(x_0)g(da^2,da^3){\rm tr}[id]area(S_n)\nonumber\\
&=\frac{4}{3}\sum_\alpha R(\nabla a^1,e_\alpha,\nabla a^4,e_\alpha)g(da^2,da^3){\rm tr}[id]area(S_n).
\end{align*}
$\mathbf{(iv-1-d)}$
\begin{align*}
&\int_{|\xi|=1}{\rm tr}\bigg(\frac{1}{2}\sum_{jr}\partial_{x_j}(a^1)\partial_{x_r}(a^4)\sum_{st}R_{jrts}(x_0)c(e_s)c(e_t)c(da^2)(da^3)|\xi|^{-2m}\bigg)\sigma(\xi)\nonumber\\
&=-\sum_{jrts}\partial_{x_j}(a^1)\partial_{x_r}(a^4)R_{jrts}(x_0)e_s(a^2)e_t(a^3){\rm tr}[id]area(S_n)\nonumber\\
&=-R(\nabla(a^1),\nabla(a^4),\nabla(a^3),\nabla(a^2)){\rm tr}[id]area(S_n).
\end{align*}
$\mathbf{(iv-1-e)}$
\begin{align*}
&\int_{|\xi|=1}{\rm tr}\bigg(2\sum_{jr}\partial_{x_j}(a^1)\partial_{x_j}\partial_{x_r}^2(a^4)c(da^2)c(da^3)|\xi|^{-2m}\bigg)\sigma(\xi)\nonumber\\
&=-2\sum_{jr}\partial_{x_j}(a^1)\partial_{x_j}\partial_{x_r}^2(a^4)g(da^2,da^3){\rm tr}[id]area(S_n)\nonumber\\
&=\Big(2\nabla(a^1)[\Delta(a^4)]-\frac{4}{3}\sum_\alpha R(\nabla (a^1),e_\alpha,\nabla (a^4),e_\alpha)\Big)g(da^2,da^3){\rm tr}[id]area(S_n).
\end{align*}
Therefore, we get
\begin{align*}
&\int_{|\xi|=1}{\rm tr}\bigg(\sigma_0(A_3B_3)\sigma_{-2m}(D^{-2m})(x_0)\bigg)\sigma(\xi)\nonumber\\
&=\Big(-\Delta(a^1)\Delta(a^4)g(da^2,da^3)-2\Delta(a^1)g(\nabla (a^4),\nabla g(da^2,da^3))-R(\nabla(a^1),\nabla(a^4),\nabla(a^3),\nabla(a^2))\nonumber\\
&+2\nabla(a^1)[\Delta(a^4)]g(da^2,da^3)\Big){\rm tr}[id]area(S_n).
\end{align*}
$\mathbf{(iv-2)}$ For $\sigma_1(A_3B_3)\sigma_{-2m-1}(D^{-2m})$:\\
By (\ref{a33}), we get
\begin{align*}
\int_{|\xi|=1}{\rm tr}\bigg(\sigma_1(A_3B_3)\sigma_{-2m-1}(D^{-2m})(x_0)\bigg)\sigma(\xi)=0.
\end{align*}
$\mathbf{(iv-3)}$ For $\sigma_2(A_3B_3)\sigma_{-2m-2}(D^{-2m})$:\\
 \begin{align}\label{A1}
&\sigma_2(A_3B_3)\sigma_{-2m-2}(D^{-2m})(x_0)\nonumber\\
&=m\sum_{jl}\partial_{x_j}(a^1)\partial_{x_l}(a^4)\xi_j\xi_lc(da^2)c(da^3)|\xi|^{-2m-2}s\nonumber\\
&-\frac{4m(m+1)}{3}\sum_{jl}\partial_{x_j}(a^1)\partial_{x_l}(a^4)\sum_{\mu\nu\alpha}R_{\mu\alpha\nu\alpha}(x_0)\xi_j\xi_l\xi_\mu\xi_\nu c(da^2)c(da^3)|\xi|^{-2m-4}.
\end{align}
$\mathbf{(iv-3-a)}$
\begin{align*}
&\int_{|\xi|=1}m{\rm tr}\bigg(\sum_{jl}\partial_{x_j}(a^1)\partial_{x_l}(a^4)\xi_j\xi_lc(da^2)c(da^3)|\xi|^{-2m-2}s\bigg)\sigma(\xi)\nonumber\\
&=-\frac{1}{2}\sum_{j}\partial_{x_j}(a^1)\partial_{x_j}(a^4)area(S_n)g(da^2,da^3)s{\rm tr}[id]\nonumber\\
&=-\frac{1}{2}g(\nabla(a^1),\nabla(a^4))area(S_n)g(da^2,da^3)s{\rm tr}[id].
\end{align*}
$\mathbf{(iv-3-b)}$
\begin{align*}
&\int_{|\xi|=1}-\frac{4m(m+1)}{3}{\rm tr}\bigg(\sum_{jl}\partial_{x_j}(a^1)\partial_{x_l}(a^4)\sum_{\mu\nu\alpha}R_{\mu\alpha\nu\alpha}(x_0)\xi_j\xi_l\xi_\mu\xi_\nu c(da^2)c(da^3)|\xi|^{-2m-4}\bigg)\sigma(\xi)\nonumber\\
&=\frac{1}{3}\bigg(\sum_{j\mu\alpha}\partial_{x_j}(a^1)\partial_{x_j}(a^4)R_{\mu\alpha\mu\alpha}(x_0)+2\sum_{jl \alpha}\partial_{x_j}(a^1)\partial_{x_l}(a^4)R_{j\alpha l\alpha}(x_0)\bigg)area(S_n)g(da^2,da^3){\rm tr}[id]\nonumber\\
&=\frac{1}{3}\bigg(\sum_{\mu\alpha}g(\nabla(a^1),\nabla(a^4))R_{\mu\alpha\mu\alpha}(x_0)+2\sum_{\alpha }R(\nabla(a^1),e_\alpha,\nabla(a^4),e_\alpha)\bigg)area(S_n)g(da^2,da^3){\rm tr}[id]\nonumber\\
&=\frac{1}{3}\bigg(g(\nabla(a^1),\nabla(a^4))s+2\sum_{\alpha }R(\nabla(a^1),e_\alpha,\nabla(a^4),e_\alpha)\bigg)area(S_n)g(da^2,da^3){\rm tr}[id].
\end{align*}
Therefore, we get
\begin{align*}
&\int_{|\xi|=1}{\rm tr}\bigg(\sigma_2(A_3B_3)\sigma_{-2m-2}(D^{-2m})(x_0)\bigg)\sigma(\xi)\nonumber\\
&=\bigg(-\frac{1}{6}g(\nabla(a^1),\nabla(a^4))s+\frac{2}{3}\sum_{\alpha }R(\nabla(a^1),e_\alpha,\nabla(a^4),e_\alpha)\bigg)area(S_n)g(da^2,da^3){\rm tr}[id].
\end{align*}
$\mathbf{(iv-4)}$ For $-i\sum_{j=1}^n\partial_{\xi_j}[\sigma_2(A_3B_3)]\partial_{x_j}[\sigma_{-2m-1}(D^{-2m})]$:
 \begin{align*}
&-i\sum_{j=1}^n\partial_{\xi_j}[\sigma_2(A_3B_3)]\partial_{x_j}[\sigma_{-2m-1}(D^{-2m})](x_0)\nonumber\\
&=\frac{8m}{3}\sum_{jplk\alpha}\partial_{x_p}(a^1)\partial_{x_l}(a^4)[\partial_{\xi_j}(\xi_p)\xi_l\xi_k+\partial_{\xi_j}(\xi_l)\xi_p\xi_k]R_{j\alpha k\alpha}(x_0)c(da^2)c(da^3)|\xi|^{-2m-2}\nonumber\\
&-m\sum_{jplkts}\partial_{x_p}(a^1)\partial_{x_l}(a^4)[\partial_{\xi_j}(\xi_p)\xi_l\xi_k+\partial_{\xi_j}(\xi_l)\xi_p\xi_k]R_{jkts}(x_0)c(da^2)c(da^3)c(e_s)c(e_t)|\xi|^{-2m-2}\nonumber\\
&+\frac{16}{3}m\sum_{jplr\alpha\beta}\partial_{x_p}(a^1)\partial_{x_l}(a^4)R_{\alpha j\beta r}(x_0)[\partial_{\xi_j}(\xi_p)\xi_l\xi_r\xi_\alpha\xi_\beta+\partial_{\xi_j}(\xi_l)\xi_p\xi_r\xi_\alpha\xi_\beta]c(da^2)c(da^3)|\xi|^{-2m-4}\nonumber\\
&+\frac{16(m-k-1)}{3}\sum_{k=0}^{m-2}\sum_{\mu=1}^{2m}\sum_{jplr\alpha\beta}\partial_{x_p}(a^1)\partial_{x_l}(a^4)R_{\alpha j\beta \mu}(x_0)[\partial_{\xi_j}(\xi_p)\xi_l\xi_\mu\xi_\alpha\xi_\beta+\partial_{\xi_j}(\xi_l)\xi_p\xi_\mu\xi_\alpha\xi_\beta]c(da^2)c(da^3)|\xi|^{-2m-4}.
\end{align*}
$\mathbf{(iv-4-a)}$
\begin{align*}
&\int_{|\xi|=1}{\rm tr}\bigg(\frac{8m}{3}\sum_{jplk\alpha}\partial_{x_p}(a^1)\partial_{x_l}(a^4)[\partial_{\xi_j}(\xi_p)\xi_l\xi_k+\partial_{\xi_j}(\xi_l)\xi_p\xi_k]R_{j\alpha k\alpha}(x_0)c(da^2)c(da^3)|\xi|^{-2m-2}\bigg)\sigma(\xi)\nonumber\\
&=-\frac{8}{3}\sum_{pk\alpha}\partial_{x_p}(a^1)\partial_{x_k}(a^4)R_{p\alpha k\alpha}(x_0) area(S_n)g(da^2,da^3){\rm tr}[id]\nonumber\\
&=-\frac{8}{3}\sum_{\alpha}R(\nabla(a^1),e_\alpha,\nabla(a^4),e_\alpha) area(S_n)g(da^2,da^3){\rm tr}[id].
\end{align*}
$\mathbf{(iv-4-b)}$
\begin{align*}
&\int_{|\xi|=1}{\rm tr}\bigg(-m\sum_{jplkts}\partial_{x_p}(a^1)\partial_{x_l}(a^4)[\partial_{\xi_j}(\xi_p)\xi_l\xi_k+\partial_{\xi_j}(\xi_l)\xi_p\xi_k]R_{jkts}(x_0)c(da^2)c(da^3)c(e_s)c(e_t)|\xi|^{-2m-2}\bigg)\sigma(\xi)\nonumber\\
&=0.
\end{align*}
$\mathbf{(iv-4-c)}$
\begin{align*}
&\int_{|\xi|=1}{\rm tr}\bigg(\sum_{jplr\alpha\beta}\partial_{x_p}(a^1)\partial_{x_l}(a^4)R_{\alpha j\beta r}(x_0)[\partial_{\xi_j}(\xi_p)\xi_l\xi_r\xi_\alpha\xi_\beta+\partial_{\xi_j}(\xi_l)\xi_p\xi_r\xi_\alpha\xi_\beta]c(da^2)c(da^3)|\xi|^{-2m-4}\bigg)\sigma(\xi)\nonumber\\
&=-\bigg(\sum_{jlr\alpha\beta}\partial_{x_j}(a^1)\partial_{x_l}(a^4)R_{\alpha j\beta r}(x_0)I_{S_n}^{lr\alpha\beta}+\sum_{jr\alpha\beta}\partial_{x_p}(a^1)\partial_{x_j}(a^4)R_{\alpha j\beta r}(x_0)I_{S_n}^{pr\alpha\beta}\bigg)g(da^2,da^3){\rm tr}[id]\nonumber\\
&=0.
\end{align*}
$\mathbf{(iv-4-d)}$\\
Similarly, we get
\begin{align*}
&\int_{|\xi|=1}{\rm tr}\bigg(\sum_{jplr\alpha\beta}\partial_{x_p}(a^1)\partial_{x_l}(a^4)R_{\alpha j\beta \mu}(x_0)[\partial_{\xi_j}(\xi_p)\xi_l\xi_\mu\xi_\alpha\xi_\beta+\partial_{\xi_j}(\xi_l)\xi_p\xi_\mu\xi_\alpha\xi_\beta]c(da^2)c(da^3)|\xi|^{-2m-4}\bigg)\sigma(\xi)=0.
\end{align*}
Therefore, we get
\begin{align*}
&\int_{|\xi|=1}{\rm tr}\bigg(-i\sum_{j=1}^n\partial_{\xi_j}[\sigma_2(A_3B_3)]\partial_{x_j}[\sigma_{-2m-1}(D^{-2m})](x_0)\bigg)\sigma(\xi)\nonumber\\
&=-\frac{8}{3}\sum_{\alpha}R(\nabla(a^1),e_\alpha,\nabla(a^4),e_\alpha) area(S_n)g(da^2,da^3){\rm tr}[id].
\end{align*}
$\mathbf{(iv-5)}$ For $-i\Sigma_{j=1}^n\partial_{\xi_j}[\sigma_1(A_3B_3)]\partial_{x_j}[\sigma_{-2m}(D^{-2m})]$:\\
By (\ref{a44}), we get
\begin{align*}
\int_{|\xi|=1}{\rm tr}\bigg(-i\Sigma_{j=1}^n\partial_{\xi_j}[\sigma_1(A_3B_3)]\partial_{x_j}[\sigma_{-2m}(D^{-2m})]\bigg)(x_0)\sigma(\xi)=0.
\end{align*}
$\mathbf{(iv-6)}$ For $-\frac{1}{2}\sum_{jl}\partial_{\xi_j}\partial_{\xi_l}[\sigma_2(A_3B_3)]\partial_{x_j}\partial_{x_l}[\sigma_{-2m}(D^{-2m})]$:
\begin{align*}
&\int_{|\xi|=1}{\rm tr}\bigg(-\frac{4}{3}m\sum_{jlrs\alpha\beta}\partial_{x_r}(a^1)\partial_{x_s}(a^4)R_{\alpha j\beta l}(x_0)[\partial_{\xi_l}(\xi_r)\partial_{\xi_j}(\xi_s)+\partial_{\xi_l}(\xi_s)\partial_{\xi_j}(\xi_r)]\xi_\alpha\xi_\beta c(da^2)c(da^3)|\xi|^{-2m-2}\bigg)\sigma(\xi)\nonumber\\
&=\frac{4}{3}\sum_{jl\alpha}\partial_{x_j}(a^1)\partial_{x_l}(a^4)R_{\alpha j\alpha l}(x_0)area(S_n)g(da^2,da^3){\rm tr}[id]\nonumber\\
&=\frac{4}{3}\sum_{\alpha}R(\nabla(a^1),e_\alpha,\nabla(a^4),e_\alpha) area(S_n)g(da^2,da^3){\rm tr}[id].
\end{align*}
Therefore, we get
\begin{align*}
&\int_{|\xi|=1}{\rm tr}\bigg(-\frac{1}{2}\sum_{jl}\partial_{\xi_j}\partial_{\xi_l}[\sigma_2(A_3B_3)]\partial_{x_j}\partial_{x_l}[\sigma_{-2m}(D^{-2m})](x_0)\bigg)\sigma(\xi)\nonumber\\
&=\frac{4}{3}\sum_{\alpha}R(\nabla(a^1),e_\alpha,\nabla(a^4),e_\alpha)area(S_n)g(da^2,da^3){\rm tr}[id].
\end{align*}
Finally, we obtain
\begin{thm}
Let $P_3=[D^2,a^1][D,a^2][D,a^3][D^2,a^4]$, we have the following noncommutative integral of $P_3$ based on the Wodzicki residue
\begin{align*}
&\int\hspace{-1.05em}-a^0[D^2,a^1][D,a^2][D,a^3][D^2,a^4]D^{-2m}\nonumber\\
&=2^m\frac{2\pi^{\frac{n}{2}}}{\Gamma(\frac{n}{2})}\int_Ma_0\bigg(-\Delta(a^1)\Delta(a^4)g(da^2,da^3)-2\Delta(a^1)g(\nabla (a^4),\nabla g(da^2,da^3))-R(\nabla(a^1),\nabla(a^4),\nabla(a^3),\nabla(a^2))\nonumber\\
&+\Big(2\nabla(a^1)[\Delta(a^4)]-\frac{2}{3}\sum_{\alpha }R(\nabla(a^1),e_\alpha,\nabla(a^4),e_\alpha)
-\frac{1}{6}g(\nabla(a^1),\nabla(a^4))s\Big)g(da^2,da^3)\bigg)d{\rm Vol_M}.
\end{align*}
\end{thm}
In the same way, we take four different parts for $A$ and $B$, and then by further calculation, we get the following result.
\begin{align}
&\int\hspace{-1.05em}-a^0[D^2,a^1][D^2,a^2][D,a^3][D,a^4]D^{-2m}-\int\hspace{-1.05em}-a^0[D,a^1][D^2,a^2][D^2,a^3][D,a^4]D^{-2m}\nonumber\\
&+\int\hspace{-1.05em}-a^0[D,a^1][D,a^2][D^2,a^3][D^2,a^4]D^{-2m}-\int\hspace{-1.05em}-a^0[D^2,a^1][D,a^2][D,a^3][D^2,a^4]D^{-2m}\nonumber\\
&=2^m\frac{2\pi^{\frac{n}{2}}}{\Gamma(\frac{n}{2})}\int_Ma_0\bigg\{
-2\Delta(a^1)g(\nabla (a^2),\nabla g(a^3,a^4))-4\nabla(a^1)\nabla(a^2)[g(da^3,da^4)]-2R(\nabla(a^1),\nabla(a^2),\nabla(a^4),\nabla(a^3))\nonumber\\
&+\bigg(2\nabla(a^1)[\Delta(a^2)]-\Delta(a^1)\Delta(a^2)-\frac{1}{6}g(\nabla(a^1),\nabla(a^2))s-\frac{2}{3}\sum_{\alpha}R(\nabla(a^1),e_\alpha,\nabla(a^2),e_\alpha)\bigg)g(da^3,da^4)\nonumber\\
&+\Delta(a^2)\Delta(a^3)g(da^1,da^4)-2\Delta(a^2)g(\nabla(a^1),\nabla_{\nabla(a^3)}\nabla(a^4))+2R(\nabla(a^2),\nabla(a^3),\nabla(a^4),\nabla(a^1))\nonumber\\
&-2\nabla(a^2)[\Delta(a^3)]g(da^1,da^4)-2\Delta(a^3)g(\nabla(a^1),\nabla_{\nabla(a^2)}\nabla(a^4))+4g(\nabla(a^1),\nabla_{\nabla(a^2)}\nabla_{\nabla(a^3)}\nabla(a^4))\nonumber\\
&-4g(\nabla_{\nabla(a^1)}\nabla(a^4),\nabla_{\nabla(a^2)}\nabla(a^3))-2R(\nabla(a^1),\nabla(a^3),\nabla(a^2),\nabla(a^4))-4g(\nabla(a^1),\nabla_{\nabla_{\nabla(a^2)}\nabla(a^3)}\nabla(a^4))\nonumber\\
&+\frac{1}{6}g(\nabla(a^2),\nabla(a^3))g(da^1,da^4)s+\frac{2}{3}\sum_{\alpha }R(\nabla(a^2),e_\alpha,\nabla(a^3),e_\alpha)g(da^1,da^4)+\bigg[2\nabla(a^3)[\Delta(a^4)]\nonumber\\
&-\Delta(a^3)\Delta(a^4)-\frac{1}{6}g(\nabla(a^3),\nabla(a^4))s-\frac{2}{3}\sum_{\alpha }R(\nabla(a^3),e_\alpha,\nabla(a^4),e_\alpha)(x_0)\bigg)g(da^1,da^2)\bigg]\nonumber\\
&+\Delta(a^1)\Delta(a^4)g(da^2,da^3)+2\Delta(a^1)g(\nabla (a^4),\nabla g(da^2,da^3))-2\nabla(a^1)[\Delta(a^4)]g(da^2,da^3)\nonumber\\
&-\bigg(-\frac{2}{3}\sum_{\alpha }R(\nabla(a^1),e_\alpha,\nabla(a^4),e_\alpha)
-\frac{1}{6}g(\nabla(a^1),\nabla(a^4))s\bigg)g(da^2,da^3)\bigg\}d{\rm Vol_M}.
\end{align}
When $n=6,~~m=3$, we have $area(S_6)=\pi^3,$ then we get the following theorem.
\begin{cor}
Let $(\mathcal{A},\mathcal{H},D)$ be the spectral triple associated to a compact spin Riemannian 6-dimensional manifold $M$. Then
a Connes-Chamseddine cycle is given
 \begin{align}
& \rho(a^0,a^1,a^2,a^3,a^4)\nonumber\\
&=8\pi^3\int_Ma_0\bigg\{
-2\Delta(a^1)g(\nabla (a^2),\nabla g(a^3,a^4))-4\nabla(a^1)\nabla(a^2)[g(da^3,da^4)]-2R(\nabla(a^2),\nabla(a^3),\nabla(a^4),\nabla(a^1))\nonumber\\
&+\bigg(2\nabla(a^1)[\Delta(a^2)]-\Delta(a^1)\Delta(a^2)-\frac{1}{6}g(\nabla(a^1),\nabla(a^2))s-\frac{2}{3}\sum_{\alpha}R(\nabla(a^1),e_\alpha,\nabla(a^2),e_\alpha)\bigg)g(da^3,da^4)\nonumber\\
&+\Delta(a^2)\Delta(a^3)g(da^1,da^4)-2\Delta(a^2)g(\nabla(a^1),\nabla_{\nabla(a^3)}\nabla(a^4))+2R(\nabla(a^2),\nabla(a^3),\nabla(a^4),\nabla(a^1))\nonumber\\
&-2\nabla(a^2)[\Delta(a^3)]g(da^1,da^4)-2\Delta(a^3)g(\nabla(a^1),\nabla_{\nabla(a^2)}\nabla(a^4))+4g(\nabla(a^1),\nabla_{\nabla(a^2)}\nabla_{\nabla(a^3)}\nabla(a^4))\nonumber\\
&-4g(\nabla_{\nabla(a_1)}\nabla(a_4),\nabla_{\nabla(a_2)}\nabla(a_3))-2R(\nabla(a^1),\nabla(a^3),\nabla(a^2),\nabla(a^4))-4g(\nabla(a^1),\nabla_{\nabla_{\nabla(a^2)}\nabla(a^3)}\nabla(a^4))\nonumber\\
&+\frac{1}{6}g(\nabla(a^2),\nabla(a^3))g(da^1,da^4)s+\frac{2}{3}\sum_{\alpha }R(\nabla(a^2),e_\alpha,\nabla(a^3),e_\alpha)g(da^1,da^4)+\bigg[2\nabla(a^3)[\Delta(a^4)]\nonumber\\
&-\Delta(a^3)\Delta(a^4)-\frac{1}{6}g(\nabla(a^3),\nabla(a^4))s-\frac{2}{3}\sum_{\alpha }R(\nabla(a^3),e_\alpha,\nabla(a^4),e_\alpha)(x_0)\bigg)g(da^1,da^2)\bigg]\nonumber\\
&+\Delta(a^1)\Delta(a^4)g(da^2,da^3)+2\Delta(a^1)g(\nabla (a^4),\nabla g(da^2,da^3))-2\nabla(a^1)[\Delta(a^4)]g(da^2,da^3)\nonumber\\
&-\bigg(-\frac{2}{3}\sum_{\alpha }R(\nabla(a^1),e_\alpha,\nabla(a^4),e_\alpha)
-\frac{1}{6}g(\nabla(a^1),\nabla(a^4))s\bigg)g(da^2,da^3)\bigg\}d{\rm Vol_M}.
 \end{align}
 \end{cor}

\section{Appendix}

In this appendix, we will prove some facts used in Lemma \ref{lemkkk} and in the computation of the residue.

\begin{lem}(Also see (4.3) in \cite{D1})
\begin{align}
(1)\sigma_{-2m}(D^{-2m})&=|\xi|^{-2m};\nonumber\\
(2)\sigma_{-2m-1}(D^{-2m})&=m|\xi|^{-(2m-2)}\bigg(-i|\xi|^{-4}\xi_k(\Gamma^k-2\sigma^k)-2i|\xi|^{-6}\xi^j\xi_\alpha\xi_\beta\partial_{x_j}g^{\alpha\beta}\bigg)\nonumber\\
&+2i\sum_{k=0}^{m-2}\sum_{\mu=1}^{2m}(k+1-m)|\xi|^{-2m-4}\xi^\mu\xi_\alpha\xi_\beta\partial^x_\mu g^{\alpha\beta};\nonumber\\
(3)\sigma_{-2m-2}(D^{-2m})(x_0)&=-\frac{m}{4}|\xi|^{-2m-2}s+\frac{m(m+1)}{3}|\xi|^{-2m-4}\xi_\mu\xi_\nu R_{\mu\alpha\nu\alpha}(x_0),
\end{align}
where $s$ is the scalar curvature.
\end{lem}
\begin{proof}(1) holds obviously.

(2)By (3.8) in \cite{Wa8}, we have
\begin{align}
\sigma_{1-\bar{n}}(D^{-\bar{n}+2})=\frac{\bar{n}-2}{2}\sigma_2^{(-\frac{\bar{n}}{2}+2)}\sigma_{-3}(D^{-2})-i\sum_{k=0}^{\frac{\bar{n}}{2}-3}\partial_{\xi_\mu}\sigma_2^{(-\frac{\bar{n}}{2}+k+2)}\partial_{x_\mu}\sigma_2^{-1}(\sigma_2^{-1})^k.
\end{align}
By plugging $\bar{n}=2m+2$ into the above equation, we get
\begin{align}
\sigma_{-2m-1}(D^{-2m})=m\sigma_2^{(-m+1)}\sigma_{-3}(D^{-2})-i\sum_{k=0}^{m-3}\sum_{\mu=1}^{2m}\partial_{\xi_\mu}\sigma_2^{(-m+k+1)}\partial_{x_\mu}\sigma_2^{-1}(\sigma_2^{-1})^k.
\end{align}
Then, by
$\sigma_{-3}(D^{-2})=-i|\xi|^{-4}\xi_k(\Gamma^k-2\sigma^k)-2i|\xi|^{-6}\xi^j\xi_\alpha\xi_\beta\partial_{x_j}g^{\alpha\beta},$ $\partial_{\xi_\mu}\sigma_2^{(-\frac{\bar{n}}{2}+k+2)}=2(k+1-m)|\xi|^{-2m+2k}\xi^\mu$ and $\partial_{x_\mu}\sigma_2^{-1}=-|\xi|^{-4}\xi_\alpha\xi_\beta\partial_{x_j}g^{\alpha\beta}$ and
(2) holds.

(3)Combined with \cite{KW}, we make a new arrangement of the proof process of (4.24) in \cite{KW}. Firstly, using $\sigma^{\tilde{\Delta}^{-m}}_{-2m}\equiv\sigma_2^{-m}$ we get the recursion relations
\begin{align}\label{AA2}
\sigma^{\tilde{\Delta}^{-\frac{n}{2}+1}}_{-n}(x,\xi)&=\sum_{|\alpha|=0}^2\sum_{k=2}^{4-|\alpha|}(-i)^{|\alpha|}\frac{1}{\alpha!}\partial_{\xi}^\alpha\sigma^{\tilde{\Delta}^{-\frac{n}{2}+2}}_{|\alpha|+k-n}\partial_{x}^\alpha\sigma^{\tilde{\Delta}^{-1}}_{-k}\nonumber\\
&=\sigma^{\tilde{\Delta}^{-\frac{n}{2}+2}}_{2-n}\sigma_2^{-1}+\sigma^{\tilde{\Delta}^{-\frac{n}{2}+2}}_{3-n}\sigma^{\tilde{\Delta}^{-1}}_{-3}+\sigma^{\tilde{\Delta}^{-\frac{n}{2}+2}}_{4-n}\sigma^{\tilde{\Delta}^{-1}}_{-4}\nonumber\\
&-i\partial_{\xi_\mu}\sigma^{\tilde{\Delta}^{-\frac{n}{2}+2}}_{3-n}\partial_{x_\mu}\sigma_2^{-1}-i\partial_{\xi_\mu}\sigma^{-\frac{n}{2}+2}_{2}\partial_{x_\mu}\sigma_{-3}^{\tilde{\Delta}^{-1}}\nonumber\\
&-\frac{1}{2}\partial_{\xi_\mu}\partial_{\xi_\nu}\sigma^{-\frac{n}{2}+2}_{2}\partial_{x_\mu}\partial_{x_\nu}\sigma_2^{-1},
\end{align}
where $\tilde{\Delta}=D^2.$\\
After further iterations
\begin{align}\label{AA1}
\sigma^{\tilde{\Delta}^{-\frac{n}{2}+2}}_{3-n}(x_0,\xi)=(\frac{1}{2}n-2)\sigma_2^{-\frac{n}{2}+3}\sigma^{\tilde{\Delta}^{-1}}_{-3}(x_0,\xi).
\end{align}
Write
 \begin{align}
\sigma(\tilde{\Delta})=\sigma_2+\sigma_1+\sigma_0;
~~~\sigma(\tilde{\Delta}^{-1})=\sum^{\infty}_{j=2}b_{-j}.
\end{align}
By the composition formula of pseudodifferential operators, we have
\begin{align}
1=\sigma(\tilde{\Delta}\circ \tilde{\Delta}^{-1})&=\sum_{\alpha}\frac{(-i)^{|\alpha|}}{\alpha!}\partial^{\alpha}_{\xi}[\sigma(\tilde{\Delta})]
\partial_x^{\alpha}[\sigma(\tilde{\Delta}^{-1})],
\end{align}
so
\begin{align}\label{hhh}
b_{-2}&=\sigma_2^{-1};~~~~b_{-3}=-\sigma_2^{-1}[\sigma_1\sigma_2^{-1}-i\sum_j\partial_{\xi_j}(\sigma_2)\partial_{x_j}(b_{-2})];\nonumber\\
b_{-4}&=-\sigma_2^{-1}[\sigma_1b_{-3}+\sigma_0b_{-2}-i\sum_j\partial_{\xi_j}(\sigma_1)\partial_{x_j}(b^{-2})-i\sum_j\partial_{\xi_j}(\sigma_2)\partial_{x_j}(b_{-3})-\frac{1}{2}\sum_j\partial_{\xi_j}\partial_{\xi_k}(\sigma_2)\partial_{x_j}\partial_{x_k}(b_{-2})].
\end{align}
Uingsing  $\partial_{x_\mu}(b_{-2})(x_0)=b_{-3}(x_0)=0$, we have
\begin{align}
b_{-4}(x_0)=\sigma^{\tilde{\Delta}^{-1}}_{-4}(x_0,\xi)=-\sigma_2^{-1}[\sigma_0b_{-2}-i\sum_j\partial_{\xi_j}(\sigma_2)\partial_{x_j}(b_{-3})-\frac{1}{2}\sum_j\partial_{\xi_j}\partial_{\xi_k}(\sigma_2)\partial_{x_j}\partial_{x_k}(b_{-2})]
\end{align}
Then
\begin{align}\label{aa3}
-i\sum_j\partial_{\xi_j}(\sigma_2)\partial_{x_j}(b_{-3})=-\sigma_2\sigma^{\tilde{\Delta}^{-1}}_{-4}-\sigma_0\sigma_2^{-1}+\frac{1}{2}\sum_j\partial_{\xi_j}\partial_{\xi_k}(\sigma_2)\partial_{x_j}\partial_{x_k}(b_{-2}).
\end{align}
Inserting Eq. (\ref{AA1}) and Eq. (\ref{aa3}) in relation (\ref{AA2}) yields
\begin{align}
\sigma^{\tilde{\Delta}^{-\frac{n}{2}+1}}_{-n}(x_0,\xi)&=\sigma^{\tilde{\Delta}^{-\frac{n}{2}+1}}_{2-n}\sigma_2^{-1}+\sigma^{-\frac{n}{2}+2}_{2}\sigma_{-4}^{\tilde{\Delta}^{-1}}+(2-\frac{n}{2})\sigma^{-\frac{n}{2}+1}_{2}[-\sigma^{-\frac{n}{2}+2}_{2}\sigma_{-4}^{\tilde{\Delta}^{-1}}-\sigma_0\sigma_2^{-1}\nonumber\\
&+\frac{1}{2}\sum_j\partial_{\xi_j}\partial_{\xi_k}(\sigma_2)\partial_{x_j}\partial_{x_k}(b_{-2})]-\frac{1}{2}\sum_j\partial_{\xi_j}\partial_{\xi_k}(\sigma_2^{-\frac{n}{2}+2})\partial_{x_j}\partial_{x_k}(b_{-2})\nonumber\\
&=\sigma^{\tilde{\Delta}^{-\frac{n}{2}+1}}_{2-n}\sigma_2^{-1}+(\frac{n}{2}-1)\sigma^{-\frac{n}{2}+2}_{2}\sigma_{-4}^{\tilde{\Delta}^{-1}}+(\frac{n}{2}-2)\sigma^{-\frac{n}{2}}_{2}\sigma_0\nonumber\\
&+\frac{4}{3}(\frac{n}{2}-2)(\frac{n}{2}-1)\sigma_2^{-\frac{n}{2}}\sigma_2^{-2}\sum_{\alpha\beta\gamma\delta}R_{\alpha\beta\gamma\delta}(x_0)\xi_\alpha\xi_\beta\xi_\gamma\xi_\delta.
\end{align}
After the iteration, and by $\sum_{\alpha\beta\mu\nu}\xi_\alpha\xi_\beta\xi_\mu\xi_\nu R_{\alpha\beta\mu\nu}(x_0)=0,$ we obtain
\begin{align}
\sigma^{\tilde{\Delta}^{-\frac{n}{2}+1}}_{-n}(x_0,\xi)&=\frac{n-2}{8}[n\sigma^{-\frac{n}{2}+2}_{2}\sigma_{-4}^{\tilde{\Delta}^{-1}}(x_0,\xi)+(n-4)\sigma^{-\frac{n}{2}}_{2}\sigma_0].
\end{align}
By plugging $n=2m+2$ into the above equation and Eq. (12b) in \cite{Ka}, then in normal coordinates, (3) holds.
\end{proof}
 $\mathbf{(i-1-f)}:$
 \begin{align}\label{p1}
&\int_{|\xi|=1}{\rm tr}\bigg(2\sum_{jr}\partial_{x_j}(a^1)\partial_{x_j}\partial_{x_r}^2(a^2)c(da^3)c(da^4)|\xi|^{-2m}\bigg)\sigma(\xi)\nonumber\\
&=\bigg(2\nabla(a^1)[\Delta(a^2)](x_0)-\frac{4}{3}\sum_\alpha R(\nabla (a^1),e_\alpha,\nabla (a^2),e_\alpha)\bigg)g(da^3,da^4){\rm tr}[id]Vol_{S^n}.
\end{align}
 \begin{proof}
In the normal coordinate system, we can obtain
 \begin{align}
 \nabla(a^1)[\Delta(a^2)](x_0)&=\sum_j\partial_{x_j}(a^1)\partial_{x_j}[\Delta(a^2)](x_0)\nonumber\\
 &=\sum_j\partial_{x_j}(a^1)\partial_{x_j}[(-\sum_{\alpha\beta}g^{\alpha\beta}\partial_{x_\alpha}\partial_{x_\beta}+\sum_{\alpha\beta}g^{\alpha\beta}\Gamma^k_{\alpha\beta}\partial_{x_k})(a^2)](x_0)\nonumber\\
 &=\sum_j\partial_{x_j}(a^1)[-\sum_{\alpha\beta}\partial_{x_j}(g^{\alpha\beta})\partial_{x_\alpha}\partial_{x_\beta}-\sum_{\alpha\beta}g^{\alpha\beta}\partial_{x_j}\partial_{x_\alpha}\partial_{x_\beta}\nonumber\\
 &+g^{\alpha\beta}(x_0)\sum^{\alpha\beta jk}\partial_{x_j}(\Gamma^k_{\alpha\beta})\partial_{x_k}+\sum^{\alpha\beta jk}\Gamma^k_{\alpha\beta}\partial_{x_j}\partial_{x_k}](a^2)(x_0)\nonumber\\
 &=-\sum_{jr}\partial_{x_j}(a^1)\partial_{x_j}\partial_{x_r}^2(a^2)+\frac{2}{3}\sum_\alpha R(\nabla (a^1),e_\alpha,\nabla (a^2),e_\alpha).
 \end{align}
 According to the following formula, we have
 \begin{align}
 \sum_{jr}\partial_{x_j}(a^1)\partial_{x_j}\partial_{x_r}^2(a^2)=\frac{2}{3}\sum_\alpha R(\nabla (a^1),e_\alpha,\nabla (a^2),e_\alpha)-\nabla(a^1)[\Delta(a^2)].
 \end{align}
 Then, after the integration, (\ref{p1}) holds.
\end{proof}
$\mathbf{(i-1-g)}$:
\begin{align}\label{p2}
\int_{|\xi|=1}{\rm tr}\bigg[4\sum_{jp}\partial_{x_j}(a^1)\partial_{x_j}\bigg(\partial_{x_p}(a^2)\partial_{x_p}[c(da^3)c(da^4)]\bigg)|\xi|^{-2m}\bigg]\sigma(\xi)=-4\nabla(a^1)\nabla(a^2)[g(da^3,da^4)]{\rm tr}[id]Vol_{S^n}.
\end{align}
 \begin{proof}
Take $|\xi|=1$, we can change the order of trace and derivative. Then
 \begin{align*}
 {\rm tr}\bigg[\sum_{jp}\partial_{x_j}(a^1)\partial_{x_j}\bigg(\partial_{x_p}(a^2)\partial_{x_p}[c(da^3)c(da^4)]\bigg)|\xi|^{-2m}\bigg]&=\nabla(a^1)\sum_p\partial_{x_p}(a^2)\partial_{x_p}[-g(da^3,da^4)]{\rm tr}[id]\nonumber\\
 &=-\nabla(a^1)\nabla(a^2)[g(da^3,da^4)]{\rm tr}[id].
 \end{align*}
 Therefore, after the integration, (\ref{p2}) holds.
\end{proof}
$\mathbf{(ii-1-b)}$:
\begin{align}\label{p3}
&\int_{|\xi|=1}{\rm tr}\bigg(2\sum_{kp}\partial_{x_k}^2(a^2)\partial_{x_p}(a^3)c(da^1)\partial_{x_p}[c(da^4)]|\xi|^{-2m}\bigg)\sigma(\xi)=2\Delta(a^2)g\bigg(\nabla(a^1),\nabla_{\nabla(a^3)}[\nabla(a^4)]\bigg){\rm tr}[id]Vol_{S^n}.
\end{align}
\begin{proof}
 By $c(da^1)=\sum_\mu e_\mu(a^1)c(e_\mu)$ and $c(da^4)=\sum_\nu e_\nu(a^4)c(e_\nu)$, we have
\begin{align*}
{\rm tr}\bigg(\sum_{kp}\partial_{x_k}^2(a^2)\partial_{x_p}(a^3)c(da^1)\partial_{x_p}[c(da^4)]\bigg)(x_0)&=-\sum_{kp}\partial_{x_k}^2(a^2)\partial_{x_p}(a^3)\sum_\mu e_\mu(a^1)\partial_{x_p}[e_\mu(a^4)](x_0){\rm tr}[id]\nonumber\\
&=\Delta(a^2)\sum_p\partial_{x_p}(a^3)\sum_\mu e_\mu(a^1)\partial_{x_p}[e_\mu(a^4)](x_0){\rm tr}[id]\nonumber\\
&=\Delta(a^2)\sum_\mu e_\mu(a^1)\sum_p\partial_{x_p}(a^3)\partial_{x_p}[e_\mu(a^4)](x_0){\rm tr}[id]\nonumber\\
&=\Delta(a^2)\sum_\mu e_\mu(a^1)\nabla(a^3)[e_\mu(a^4)](x_0){\rm tr}[id].
\end{align*}
And
\begin{align*}
g(\nabla(a^1),\nabla_{\nabla(a^3)}(\nabla(a^4)))(x_0)&=g\bigg(\sum_\mu e_\mu(a^1)e_\mu,\nabla_{\nabla(a^3)}(\sum_le_l(a^4)e_l)\bigg)(x_0)\nonumber\\
&=g\bigg(\sum_\mu e_\mu(a^1)e_\mu,\nabla_{\nabla(a^3)}(\sum_le_l(a^4)e_l)\bigg)(x_0)\nonumber\\
&=g\bigg(\sum_\mu e_\mu(a^1)e_\mu,\sum_l\nabla(a^3)(e_l(a^4))e_l+\sum_le_l(a^4)\nabla_{\nabla(a^3)}e_l\bigg)(x_0)\nonumber\\
&=\sum_{\mu } e_\mu(a^1)\nabla(a^3)[e_\mu(a^4)].
\end{align*}
Then
\begin{align*}
\Delta(a^2)\sum_\mu e_\mu(a^1)\nabla(a^3)[e_\mu(a^4)](x_0){\rm tr}[id]=\Delta(a^2)g\bigg(\nabla(a^1),\nabla_{\nabla(a^3)}[\nabla(a^4)]\bigg){\rm tr}[id].
\end{align*}
Therefore, after the integration, (\ref{p3}) holds.
\end{proof}
$\mathbf{(ii-1-g)}$:
\begin{align}\label{p5}
&\int_{|\xi|=1}{\rm tr}\bigg(4\sum_{jp}\partial_{x_j}(a^2)\partial_{x_p}(a^3)c(da^1)\partial_{x_j}\partial_{x_p}[c(da^4)]|\xi|^{-2m}\bigg)\sigma(\xi)\nonumber\\
&=-4\Big(g(\nabla(a^1),\nabla_{\nabla(a^2)}\nabla_{\nabla(a^3)}\nabla(a^4))-g(\nabla_{\nabla(a_1)}\nabla(a_4),\nabla_{\nabla(a_2)}\nabla(a_3))\nonumber\\
&-\frac{1}{2}R(\nabla(a^2),\nabla(a^3),\nabla(a^1),\nabla(a^4))\Big){\rm tr}[id]Vol_{S^n}.
\end{align}
 \begin{proof}
 By $c(da^1)=\sum_\mu e_\mu(a^1)c(e_\mu)$ and $c(da^4)=\sum_\nu e_\nu(a^4)c(e_\nu)$, we have
 \begin{align*}
 {\rm tr}\bigg(\sum_{jp}\partial_{x_j}(a^2)\partial_{x_p}(a^3)c(da^1)\partial_{x_j}\partial_{x_p}[c(da^4)]\bigg)&=-\sum_{jp}\partial_{x_j}(a^2)\partial_{x_p}(a^3)\sum_\mu e_\mu(a^1)\partial_{x_j}\partial_{x_p}[ e_\mu(a^4)]{\rm tr}[id].
 \end{align*}
 Moreover
 \begin{align*}
 &g\bigg(\nabla(a^1),\nabla_{\nabla(a^2)}\nabla_{\nabla(a^3)}\nabla(a^4)\bigg)(x_0)\nonumber\\
 &=g\bigg(\nabla(a^2),\nabla_{\nabla(a^1)}\nabla_{\nabla(a^3)}\Big(\sum_\mu e_\mu(a^4)e_\mu\Big)\bigg)(x_0)\nonumber\\
 &= g\bigg(\nabla(a^2),\sum_\mu\nabla_{\nabla(a^1)}[\nabla(a^3)e_\mu(a^4)]e_\mu+\sum_\mu e_\mu(a^4)\nabla_{\nabla(a^3)}e_\mu\bigg)(x_0)\nonumber\\
 &= g\bigg(\nabla(a^2),\sum_\mu\nabla(a^1)[\nabla(a^3)e_\mu(a^4)]e_\mu\bigg)(x_0)+g\bigg(\nabla(a^2),\sum_\mu \nabla(a^1)e_\mu(a^4)\nabla_{\nabla(a^3)}e_\mu\bigg)(x_0)\nonumber\\
&+g\bigg(\nabla(a^2),\sum_\mu \nabla(a^3)e_\mu(a^4)\nabla_{\nabla(a^1)}e_\mu\bigg)(x_0)+g\bigg(\nabla(a^2),\sum_\mu e_\mu(a^4)\nabla_{\nabla(a^1)}\nabla_{\nabla(a^3)}e_\mu\bigg)(x_0)\nonumber\\
&= g\bigg(\nabla(a^2),\sum_\mu\nabla(a^1)(\nabla(a^3)e_\mu(a^4))e_\mu\bigg)(x_0)+g\bigg(\nabla(a^2),\sum_\mu e_\mu(a^4)\nabla_{\nabla(a^1)}\nabla_{\nabla(a^3)}e_\mu\bigg)(x_0),
 \end{align*}
 where
 \begin{align*}
 &g\bigg(\nabla(a^2),\sum_\mu\nabla(a^1)(\nabla(a^3)e_\mu(a^4))e_\mu\bigg)(x_0)\nonumber\\
 &=g\bigg(\partial_{x_j}(a^2)e_j,\sum_{\mu lq}\partial_{x_1}(a^1)\partial_{x_l}[\partial_{x_q}(a^3)\partial_{x_q}(e_\mu(a^4))]e_\mu\bigg)(x_0)\nonumber\\
 &=\sum_{jlp}\partial_{x_j}(a^2)\partial_{x_l}(a^1)\partial_{x_q}(a^3)\partial_{x_l}\partial_{x_q}[ e_j(a^4)]+\sum_{jlq}\partial_{x_j}(a^1)\partial_{x_l}(a^2)\partial_{x_q}\partial_{x_l}(a^3)\partial_{x_j}\partial_{x_q}(a^4)\nonumber\\
 &=\sum_{jlp}\partial_{x_j}(a^1)\partial_{x_l}(a^2)\partial_{x_q}(a^3)\partial_{x_l}\partial_{x_q}[ e_j(a^4)]+\sum_{jlq}\partial_{x_j}(a^1)\partial_{x_l}(a^2)\partial_{x_q}\partial_{x_l}(a^3)\partial_{x_j}\partial_{x_q}(a^4);
 \end{align*}
 \begin{align*}
 &g\bigg(\nabla_{\nabla(a^1)}\nabla(a^4),\nabla_{\nabla(a^2)}\nabla(a^3)\bigg)=\sum_{jlq}\partial_{x_j}(a^1)\partial_{x_l}(a^2)\partial_{x_q}\partial_{x_l}(a^3)\partial_{x_j}\partial_{x_q}(a^4);
 \end{align*}
 \begin{align*}
 &g\bigg(\nabla(a^2),\sum_\mu e_\mu(a^4)\nabla_{\nabla(a^1)}\nabla_{\nabla(a^3)}e_\mu\bigg)(x_0)\nonumber\\
 &=g\bigg(\nabla(a^2),\sum_\mu e_\mu(a^4)\nabla_{\nabla(a^1)}\sum_{jl}g^{jl}\frac{\partial a^3}{\partial_{x_j}}\nabla_{\partial_{x_l}}e_\mu\bigg)\nonumber\\
 &=g\bigg(\nabla(a^2),\sum_\mu e_\mu(a^4)\sum_{j}\frac{\partial a^3}{\partial_{x_j}}(x_0)\nabla_{\nabla(a^1)}\nabla_{\partial_{x_j}}e_\mu\bigg)\nonumber\\
 &=g\bigg(\nabla(a^2),\sum_\mu e_\mu(a^4)\sum_{jl}\frac{\partial a^3}{\partial_{x_j}}\frac{\partial a^1}{\partial_{x_l}}\sum_\alpha(\partial_{x_l}H_{j\alpha\mu}e_\alpha+H_{j\alpha\mu}\nabla_{\partial_{x_l}}e_\alpha)(x_0)\bigg)\nonumber\\
 &=\frac{1}{2}\sum_{\mu jl\alpha}e_\mu(a^4)e_j(a^3)e_l(a^1)R_{lj\alpha\mu}(x_0)e_\alpha(a^2)\nonumber\\
 &=\frac{1}{2}R(\nabla(a^1),\nabla(a^3),\nabla(a^2),\nabla(a^4)),
 \end{align*}
 where $H_{j\alpha\mu}=g(\nabla_{\partial_{x_j}}e_\mu,e_\alpha)$. Therefore
 \begin{align*}
 &\sum_{jp}\partial_{x_j}(a^2)\partial_{x_p}(a^3)\sum_\mu e_\mu(a^1)\partial_{x_j}\partial_{x_p}[ e_\mu(a^4)]\nonumber\\
 &=g(\nabla(a^1),\nabla_{\nabla(a^2)}\nabla_{\nabla(a^3)}\nabla(a^4))-g(\nabla_{\nabla(a_1)}\nabla(a_4),\nabla_{\nabla(a_2)}\nabla(a_3))-\frac{1}{2}R(\nabla(a^1),\nabla(a^3),\nabla(a^2),\nabla(a^4)).
 \end{align*}
After the integration, (\ref{p5}) holds.
\end{proof}
$\mathbf{(ii-1-h)}$:
\begin{align}\label{p6}
&\int_{|\xi|=1}{\rm tr}\bigg(4\sum_{jp}\partial_{x_j}(a^2)\partial_{x_j}\partial_{x_p}(a^3)c(da^1)\partial_{x_p}[c(da^4)]|\xi|^{-2m}\bigg)\sigma(\xi)=-4g(\nabla(a^1),\nabla_{\nabla_{\nabla(a^2)}\nabla(a^3)}\nabla(a^4)){\rm tr}[id]Vol_{S^n}.
\end{align}
\begin{proof}
By $c(da^1)=\sum_\mu e_\mu(a^1)c(e_\mu)$ and $c(da^4)=\sum_\nu e_\nu(a^4)c(e_\nu)$, we have
\begin{align*}
{\rm tr}\bigg(\sum_{jp}\partial_{x_j}(a^2)\partial_{x_j}\partial_{x_p}(a^3)c(da^1)\partial_{x_p}[c(da^4)]\bigg)(x_0)&=-\sum_{jp}\partial_{x_j}(a^2)\partial_{x_j}\partial_{x_p}(a^3)\sum_\mu e_\mu(a^1)\partial_{x_p}[e_\mu(a^4)](x_0){\rm tr}[id].
\end{align*}
Moreover
\begin{align*}
\nabla_{\nabla(a^2)}\nabla(a^3)(x_0)&=\sum_je_j(a^2)\nabla_{e_j}[\sum_le_l(a^3)e_l](x_0)\nonumber\\
&=\sum_je_j(a^2)\sum_le_je_l(a^3)e_l(x_0)\nonumber\\
&=\sum_j\partial_{x_j}(a^2)\sum_l\partial_{x_j}\partial_{x_l}(a^3)\partial_{x_l}(x_0).
\end{align*}
Then
\begin{align*}
\sum_{jp}\partial_{x_j}(a^2)\partial_{x_j}\partial_{x_p}(a^3)\sum_\mu e_\mu(a^1)\partial_{x_p}[e_\mu(a^4)](x_0){\rm tr}[id]&=\sum_{jp}\partial_{x_j}(a^2)\partial_{x_j}\partial_{x_p}(a^3)g\bigg(\nabla(a^1),\nabla_{\partial_{x_p}}\nabla(a^4)\bigg)(x_0){\rm tr}[id]\nonumber\\
&=g(\nabla(a^1),\nabla_{\nabla_{\nabla(a^2)}\nabla(a^3)}\nabla(a^4)){\rm tr}[id].
\end{align*}
After the integration, (\ref{p6}) holds.
\end{proof}

\section*{ Declarations}
\textbf{Ethics approval and consent to participate:} Not applicable.

\textbf{Consent for publication:} Not applicable.

\textbf{Availability of data and materials:} The authors confrm that the data supporting the fndings of this study are available within the article.

\textbf{Competing interests:} The authors declare no competing interests.

\textbf{Funding:} This research was funded by National Natural Science Foundation of China: No.11771070, No.12401059 and the Fundamental Research Funds for the Central Universities N2405015.

\textbf{Author Contributions:} All authors contributed to the study conception and design. Material preparation,
data collection and analysis were performed by TW and YW. The frst draft of the manuscript was written
by TW and all authors commented on previous versions of the manuscript. All authors read and approved
the final manuscript.
\section*{Acknowledgements}
This work was supported by NSFC. 11771070, NSFC. 12401059 and the Fundamental Research Funds for the Central Universities N2405015. The authors thank the referee for his (or her) careful reading and helpful comments.

\end{document}